\documentclass[reqno]{amsart}
\usepackage{amssymb,latexsym,amsmath,amsthm,enumerate,tikz,hyperref}
\usepackage[mathscr]{eucal}
\usepackage{algorithm}
\usepackage{algorithmic,mathtools,orcidlink}

\theoremstyle{plain}
\newtheorem{thm}{Theorem}[section]
\newtheorem{lem}[thm]{Lemma}
\newtheorem{cor}[thm]{Corollary}
\newtheorem{pro}[thm]{Proposition}
\newtheorem{example}[thm]{Example}
\newtheorem{eg}[thm]{Example}

\newtheorem{BFBT}[thm]{Birkhoff's Finite Basis Theorem}

\theoremstyle{definition}
\newtheorem{defn}[thm]{Definition}

\newtheorem{question}[thm]{Question}

\newcommand{\varmem}{\textsf{var-mem}}

\newcommand{\varsimem}{\textsf{var-si-mem}}
\newcommand{\congclass}{\textsf{cong-class}}
\newcommand{\collapse}{\looparrowright}
\newcommand{\qr}{\operatorname{qr}}
\newcommand{\ps}{\operatorname{ps}}
\newcommand{\tp}{\operatorname{tp}}

\newcommand{\divides}{\rightsquigarrow}

\newcommand{\A}{\mathbf A}
\newcommand{\B}{\mathbf B}

\newcommand{\CSP}{\operatorname{CSP}}

\newcommand{\up}[1]{\textup{#1}}

\newcommand{\sg}[2]{\mathrm{sg}_{#1}(#2)}

\newcommand{\tr}{\triangleright}

\newcommand{\FO}{\operatorname{FO}}

\newcommand{\Mod}{\operatorname{Mod}}
\newcommand{\Th}{\operatorname{Th}}
\newcommand{\Cg}{\operatorname{Cg}}

\newcommand{\bigand}{\operatornamewithlimits{\hbox{\Large$\&$}}}

\begin{document}

\title[Preservation, definability and complexity]{Finite model theory for pseudovarieties and universal algebra: preservation, definability and complexity.}

\author{Lucy Ham\,\orcidlink{0000-0002-5168-1083}}
\address{MACSYS, School of Mathematics and Statistics and School of Biosciences, University of Melbourne, Victoria 3010
Australia} \email{lucy.ham@unimelb.edu.au}
\author{Marcel Jackson\,\orcidlink{0000-0002-8149-1141}}
\address{Department of Mathematical and Physical Sciences, La Trobe University, Victoria  3086,
Australia} \email{m.g.Jackson@latrobe.edu.au}

\subjclass[2020]{Primary: 03C13, 08B05; Secondary: 08B26, 68Q15, 68Q19, 08A30, 20M07}
\keywords{}

\begin{abstract}
We explore new interactions between finite model theory and classical streams of universal algebra and semigroup theory.  A key result is an example of finite algebras whose variety is not finitely axiomatisable in first order logic, but where the class of finite members is finitely axiomatisable amongst finite algebras.  
These algebras present a negative solution to a first order formulation of the Eilenberg-Sch\"utzenberger problem, and witness the simultaneous failure of the {\L}os-Tarski Theorem, the SP-Preservation Theorem and Birkhoff's HSP-Preservation Theorem at the finite level.  
The examples also show that a pseudovariety without any finite pseudoequational basis may be finitely axiomatisable in first order logic amongst finite algebras.  
Other results include the undecidability of deciding first order definability of the pseudovariety of a finite algebra, and a mapping from any nontrivial fixed finite template constraint satisfaction problem to a first order equivalent variety membership problem.
\end{abstract}

\maketitle

\section{Introduction}
The central themes of finite model theory have primarily concerned relational signatures, but several facets of reasonably independent lineage have algebraic origins: tame congruence theory in universal algebra \cite{hobmck} and the theory of pseudovarieties in semigroup theory \cite{alm,eil} are two prominent examples, both with very significant development and many deep and difficult results. 
The present article aims to develop some stronger cross-fertilisation between the two genetic lines: relational and algebraic.  The article will provide some initial development of several classical finite model theoretic ideas in universal algebraic settings, but also  explore some classical universal algebraic considerations in a new light.  The article is structured with these dual goals in mind and intentionally written in a way that hopes to be approachable to both an algebraic audience and a finite model theory audience.

Many of the contributions relate to \emph{varieties} of algebras: classes of algebras, all of the same signature, that are closed under taking homomorphic images ($\mathsf{H}$), isomorphic copies of subalgebras ($\mathsf{S}$) and direct products ($\mathsf{P}$).  
Equivalently, varieties are the classes that can be defined as the class of models of a set of \emph{equations} (also sometimes called \emph{identities}), meaning universally quantified equalities between algebraic terms.
\emph{Pseudovarieties} are the finite counterpart: classes of finite algebras closed under $\mathsf{H}$, $\mathsf{S}$ and finitary direct products $\mathsf{P}_{\rm fin}$.
The original contributions of the article are sorted into a series of ``Toolkit'' sections and of ``Result'' sections.  
In the toolkit sections we provide some initial development of  basic finite model theoretic considerations with applicability to  considerations in universal algebra, particularly relating to variety membership.  
The list of possible developments is far from complete, but is sufficient for the needs of the results sections and contains a few incidental results as well. 
\begin{itemize}
\item In Section \ref{sec:si} we explore congruence generation and subdirect irreducibility, solving a problem of Bergman and Slutzki \cite{berslu02} and introducing the notion of \emph{definable completely meet irreducible congruences} and \emph{first order subdirect decomposition}.
\item In Section \ref{sec:preservation} we provide some preservation theorems that hold at the finite level in algebraic settings.  An example is Corollary~\ref{cor:compton}, which shows that if a finitely generated pseudovariety is definable by a $\forall^*\exists^*$ sentence amongst finite algebras, then it is the finite part of a finitely based variety, and so definable by a finite set of equations.
\item In Section \ref{sec:EF} we explore Ehrenfeucht-Fra\"{\i}ss\'e games in algebraic signatures, providing a completeness theorem in uniformly locally finite settings.
\end{itemize}
The ``Result'' sections use ideas from the toolkit sections to provide some new finite model theoretic results for universal algebra.
\begin{itemize}
\item In Section \ref{sec:FO} we provide examples of finite algebras whose pseudovariety is first order definable by a $\forall^*\exists^*\forall^*$ sentence, yet not finitely based in the conventional sense: there is no finite axiomatisation by equations, nor even by pseudoequations.  
In other words, they provide a negative solution to a first order logic formulation of the long-standing Eilenberg-Sch\"utzenberger problem~\cite{eilsch} (which remains open in original formulation and described as ``diabolical'' by Shaheen and Willard~\cite{shawil}), and witness the simultaneous failure of multiple classical preservation theorems.  
The examples also provide a tight upper limit to any further extensions of the aforementioned Corollary \ref{cor:compton}.  
These results showcase toolkit Sections~\ref{sec:si} and \ref{sec:preservation}.
\item In Section \ref{sec:comp}, we show how to map the complexity-theoretic landscape of fixed finite template constraint satisfaction problems precisely into the pseudovariety membership problem for finite algebras: for every finite relational structure~$\mathbb{S}$ we find a finite algebra ${\bf S}$ such that the constraint satisfaction problem $\CSP(\mathbb{S})$ for $\mathbb{S}$ is first order equivalent to the  problem of deciding membership in the variety of~$\mathbf{S}$.  
Toolkit Sections~\ref{sec:si} and \ref{sec:preservation} are again utilised.
\item In Section \ref{sec:undecid} we use material from the toolkit Section~\ref{sec:EF} to provide a proof that McKenzie's $A(\mathcal{T})$ construction from \cite{mck2} yields the undecidability of the finite model theoretic version of Tarski's Finite Basis Problem: the proof uses a result from Section~\ref{sec:FO} to bypass both McKenzie's $F(\mathcal{T})$ construction~\cite{mck3} and Willard's Finite Basis Theorem \cite{wil2}.
\end{itemize}
As the first item in the results collection is perhaps the centrepiece of the paper, we give a little more detail here; full details are in Section~\ref{sec:FO}.
Let ${\bf S}$ be any finite algebra whose operations include a semilattice operation, and for which there is a projection term: a term $t(x,y)$ explicitly in two variables but satisfying $t(x,y)\approx y$.  
Any finite lattice is an example, as the absorption law provides the projection term and either meet or join serves as the semilattice operation.
Let $\flat({\bf S})$ denote the flat extension of ${\bf S}$: the result of adding a new point $0$ that is absorbing in all operations of ${\bf S}$, 
as well as a \emph{further} semilattice meet operation in which all distinct elements meet to~$0$.
We show in Theorem~\ref{thm:FO} that the pseudovariety generated by $\flat({\bf S})$ is always equal to the class of finite models of some $\forall^*\exists^*\forall^*$ sentence.
However if~${\bf S}$ has no finite basis for its universal Horn theory, then the variety generated by $\flat({\bf S})$ is not finitely axiomatisable in first order logic (equivalently, has no finite basis for its equational theory).
Moreover, in all known examples of finite algebras where the universal Horn sentences of ${\bf S}$ are without a finite basis, this can be witnessed by finite algebras ${\bf S}_1,{\bf S}_2,\dots$ in the sense that each~${\bf S}_n$ lies outside of the universal Horn class of ${\bf S}$ (so fails some universal Horn sentence satisfied by ${\bf S}$), but the $n$-generated subalgebras of~${\bf S}_n$ lie within the universal Horn class of ${\bf S}$ (so satisfy all $n$-variable universal Horn sentences).
In this situation, we obtain the stronger property that the class of finite algebras in the variety of~$\flat({\bf S})$ has no finite basis for its equations, nor indeed its pseudoequations.
Belkin~\cite{bel}, and later Tumanov~\cite{tum}, show that there are finite lattices with the nonfinite basis property for their universal Horn class, witnessed by finite algebras in the sense just described.  
The simplest is $M_{3-3}$ with~10 elements (so that $\flat(M_{3-3})$ has 11 elements).

The finite basis problem for finitely generated varieties has been a central theme in both semigroup theory and universal algebra for decades, with roots in B.~H.~Neumann's work on group identities in the 1930s~\cite{neu}.
The examples from Section~\ref{sec:FO} show that the familiar equivalence between finite basis property and first order definability for varieties does not extend to finitely generated pseudovarieties, despite closure under homomorphic images, subalgebras and finite products.
In view of this separation, it is striking that first order definability of finitely generated pseudovarieties has not been an accompanying theme of exploration for much of the history of these kinds of questions.

Before we commence the preliminary sections and subsequent development of original results, we provide some deeper contextual background to the developments and their motivation.  
\subsection{Computational complexity.}\label{subsec:compcomp}
A central context for the present work has been the following computational problem, where the \emph{variety generated by ${\bf B}$} denotes the smallest variety containing ${\bf B}$, namely $\mathsf{HSP}({\bf B})$, or equivalently the class $\Mod(\Th_{\rm eq}({\bf B}))$ of models of the equational theory of ${\bf B}$.\\[1mm]
\fbox{\parbox{0.9\textwidth}{\noindent $\varmem(*,*)$\\
Instance: two finite algebras ${\bf A}$ and ${\bf B}$ of the same signature.\\
Question: is ${\bf A}$ in the variety generated by ${\bf B}$?}}
\medskip

Typically we are interested in the case where ${\bf B}$ is fixed, and the input is an algebra~${\bf A}$ only; this problem is denoted by $\varmem(*,{\bf B})$, and we similarly let $\varmem({\bf A},*)$ denote the dual problem where ${\bf A}$ is fixed.  We will also be interested in a promise restriction of this problem, where the input is required to be subdirectly irreducible (that is, there is a unique minimum nontrivial congruence; see Section~\ref{sec:si}).  We denote this by $\varsimem(*,{\bf B})$.

The \varmem\ problem is, in general, doubly exponential time complete, even in the restricted $\varmem(*,{\bf B})$ form: Marcin Kozik \cite{koz09} showed there is a finite algebra~${\bf B}$, for which $\varmem(*,{\bf B})$ is 2\texttt{EXPTIME}-complete, while a doubly exponential time algorithm solving \varmem\ in full generality can be found in Bergman and Slutzki~\cite{berslu00}.  
Overall, it seems rather hard to determine the complexity of $\varmem(*,{\bf B})$ for given~${\bf B}$, and it has only been over the last couple of decades that any examples were found for which the problem was not polynomial time solvable, despite awareness of this problem at least 10 years prior (Problem 1 of \cite{alm} or Problem 3.11 of \cite{khasap} for example).
Zolt\'an Sz\'ekely found a flat graph algebra with \texttt{NP}-complete finite membership in its variety \cite{sze}.  Monoid and semigroup examples with \texttt{NP}-hard finite membership were subsequently provided by the second author and Ralph McKenzie \cite{jacmck}, eventually leading to a 6-element example by the second author \cite{jac:SAT}, the minimum possible cardinality (for a semigroup).  Kl\'{\i}ma, Kunc and Pol\'{a}k~\cite{KKP} provide a natural family of finite semigroup examples (the smallest has 42 elements), each generating a variety with co-\texttt{NP}-complete finite membership problem.  
A 26-element involuted semigroup whose variety has \texttt{NP}-hard finite membership was found by the second author and Volkov~\cite{jacvol}, though this can be reduced to 10 elements using the developments of \cite{jac:SAT}; and a $3$-element commutative and additively idempotent semiring with \texttt{NP}-complete membership has been given in~\cite{JRZ}, the smallest possible size for any algebra with intractable $\varmem$ problem. 
Marcin Kozik found several examples of increasingly harder finite membership (\texttt{PSPACE}-complete and \texttt{EXPSPACE} complete~\cite{koz07}), eventually leading to the aforementioned example with 2\texttt{EXPTIME}-complete finite membership problem for its variety~\cite{koz09}.

There is a connection between the \varmem\ problem and the finite basis problem, as when ${\bf B}$ has a finite basis for its equational theory, then one can decide $\varmem(*,{\bf B})$ by simply testing the equations in the basis on the input ${\bf A}$.  Not only is this a polynomial time solution, it shows that $\varmem(*,{\bf B})$ is first order definable, in the sense that the YES instances of the problem are precisely the finite models of a single first order sentence (the conjunction of the equations in the basis). 
This immediately motivates a finite model theoretic consideration: is this the only way that $\varmem(*,{\bf B})$ can be first order definable?  In the next section we will see that this question is the natural first order variant of the long standing Eilenberg-Sch\"utzenberger problem.

We also consider variants of the $\varmem$ problem, allowing for other combinations of the operators $\mathsf{H}$, $\mathsf{S}$ and $\mathsf{P}$, as follows.  In the following, $\mathsf{K}$ is any operator built from a combination of $\mathsf{H}$, $\mathsf{S}$ and $\mathsf{P}$, possibly with repeats.  Sometimes we also consider finitary direct products $\mathsf{P}_{\rm fin}$ as well as $\mathsf{P}^+$, meaning direct products over a nonempty set of models.\medskip

\fbox{\parbox{0.9\textwidth}{\noindent $*\in \mathsf{K}(*)$ (the $\mathsf{K}$ membership problem)\\
Instance: two finite algebras ${\bf A}$ and ${\bf B}$ of the same signature.\\
Question: is ${\bf A}\in\mathsf{K}({\bf B})$?}}
\medskip

Beyond the $\varmem(*,*)$ problem, which is $\mathsf{K}=\mathsf{HSP}$, we will primarily be interested in $\mathsf{K}=\mathsf{SP}$, but other combinations such as $\mathsf{HS}$, and $\mathsf{H}$ are of interest.  As noted in \cite[Table~1]{jacmck}, the $*\in \mathsf{K}(*)$ problem is in \texttt{NP} for each of $\mathsf{K}\in\{\mathsf{SP},\mathsf{HS},\mathsf{H},\mathsf{S}\}$, and in the uniform formulation (both algebras are allowed as input) are \texttt{NP}-complete, even within semigroups.

In general, we will say that a membership problem for a class of structures $\mathscr{K}$ is \emph{first order definable} if there is a first order sentence $\Phi$ such that ${\bf A}\in \mathscr{K}$ if and only if ${\bf A}\models\Phi$, where ${\bf A}$ ranges over finite structures.
From a complexity-theoretic perspective, first order definability is an extremely strong property, placing the problem within the complexity class~$\texttt{AC}^0$ (a proper subclass of logspace \texttt{L}) denoting base level of the alternating circuit hierarchy: polynomial sized, constant-depth unlimited fan-in circuits.  
Immerman~\cite{imm} showed that a membership problem lies in~$\texttt{AC}^0$ if and only if it is first order definable once all objects are endowed with a strict linear order and some other basic bit-arithmetic properties on the universe (assumed an initial segment of $\{0,1,2\dots\}$).

\subsection{Universal algebra.}
One motivation for examining the finer computational complexity of tractable variety membership problems is that there are connections with two of the most famous open problems in universal algebra.  
\begin{quotation}
\noindent {\em The Eilenberg-Sch\"utzenberger Problem \cite[p.~417]{eilsch}.}  (The ES problem.) Is it true that when a finite algebra ${\bf A}$ of finite signature generates a variety without a finite equational basis, then the finite part of this variety has no finite equational basis?
\end{quotation}
The ES problem was solved positively for semigroups by Mark Sapir~\cite{sap}; see also Shaheen and Willard~\cite{shawil} for further discussion and references.
If ${\bf A}$ were a counterexample  to the  Eilenberg-Sch\"utzenberger Problem, then the complexity of the finite membership problem in $\mathsf{HSP}({\bf A})$ is in \texttt{FO} (the class of first order definable decision problems), even if $\mathsf{HSP}({\bf A})$ itself is not finitely axiomatisable in first order logic.  Indeed, in the context of computational complexity the following variant might seem more pertinent.
\begin{quotation}
\noindent {\em First order ES Problem.}  Is it true that when a finite algebra ${\bf A}$ of finite signature generates a variety without a finite axiomatisation in first order logic, then the finite part of this variety has no finite axiomatisation in first order logic? 
\end{quotation}
Even though the finite axiomatisability in the premise of this implication is equivalent to finite axiomatisability by equations, the corresponding conclusion is not clear at the finite level.  
Our examples in Section~\ref{sec:FO} show that the First order ES Problem has a negative solution.  

The next problem is usually attributed to a spoken speculation by Bjarni J\'onsson at Oberwolfach in 1976 (see \cite[pp.~256]{jacmcn}) but also appears in the PhD thesis of Robert E.~Park \cite{par} from the same year.  (See \cite{wilFBP} for detailed discussion of this problem.)
\begin{quotation}
\noindent {\em The J\'onsson-Park Conjecture \cite{par}.} (JP conjecture.) Does every finite algebra of finite signature whose variety has only finitely many subdirectly irreducibles also have a finite basis for its equations?
\end{quotation}
If ${\bf A}$ were a counterexample  to the JP conjecture then the complexity of $\varmem(*,{\bf A})$ is in \texttt{P}, because any instance ${\bf B}$ may be decomposed in polynomial time to a polynomial number of subdirectly irreducible quotients \cite{DDK}, and it suffices to check each of these for membership in the finite list of subdirectly irreducibles in $\mathsf{HSP}({\bf A})$.  
We will present a universal algebraic condition that in some cases enables this decomposition to be performed as a first order construction.  
A first order variant of the JP Conjecture seems relevant, though we make no conjecture. 
\begin{quotation}
\noindent {\em First order JP Problem.} If a finite algebra of finite signature generates a variety with only finitely many subdirectly irreducibles, must its pseudovariety be definable in first order logic at the finite level?
\end{quotation}
We do not resolve this problem in the present article, though demonstrate some of the ideas presented by giving a reasonably straightforward positive verification of the First order JP Problem in the case of semilattice-based algebras.  This result follows already from the more powerful finite basis theorem for congruence meet semidistributive varieties due to Ross Willard \cite{wil2}, but our proof is very different and significantly easier.

\subsection{Preservation Theorems}\label{subsec:preservation}
The First order ES-problem is related to the programme of classifying which classical model theoretic preservation theorems hold at the finite level; see \cite{alegur} or \cite[\S2]{rosen}.   We briefly recall some of the classical preservation theorems of relevance; $\Mod(\Phi)$ will denote the models of the sentence (or set of sentences) $\Phi$.
\medskip

\noindent {\bf The {\L}os-Tarski Theorem} (Preservation Theorem for $\mathsf{S}$-closed classes).
\begin{quotation}
Let $\mathcal{L}$ be a signature and $\mathscr{K}$ be a hereditary class of $\mathcal{L}$-structures (that is, closed under taking isomorphic copies of substructures). If $\mathscr{K}=\Mod(\Phi)$ for some first order sentence $\Phi$, then $\Phi$ is equivalent to a universal sentence.
\end{quotation}
For the following recall that an existential positive sentence in first order logic is a sentence of the form $
\exists \bar{x}\phi(\bar{x})$ where $\phi(\bar{x})$ is an open formula built from atomic formul{\ae} using only conjunction and disjunction (no negation or implication).
\medskip

\noindent {\bf Homomorphism Preservation Theorem}.
\begin{quotation}
Let $\mathcal{L}$ be a signature and $\mathscr{K}$ be a class of $\mathcal{L}$-structures closed under homomorphisms: whenever ${\bf A}$ and ${\bf B}$ are $\mathcal{L}$-structures with a homomorphism from ${\bf A}$ into ${\bf B}$ and ${\bf A}\in \mathscr{K}$ then ${\bf B}\in \mathscr{K}$ also.  
If $\mathscr{K}=\Mod(\Phi)$ for some first order sentence $\Phi$, then $\Phi$ is equivalent to an existential positive sentence.
\end{quotation}
Recall that a \emph{Horn clause} is a disjunction of atomic formul{\ae} and negated atomic formul{\ae}, where at most one disjunct is not negated.  A \emph{Horn formula} is a conjunction of Horn clauses, and a \emph{universal Horn sentence} is a universally quantified Horn formula.
\medskip

\noindent {\bf Preservation Theorem for $\mathsf{SP}^+$-classes.}
\begin{quotation}
If $\Phi$ is a first order sentence defining a class $\mathscr{K}$ of structures closed under taking isomorphic copies of substructures and direct products over nonempty sets of models, then $\Phi$ is logically equivalent to a universal Horn sentence.
\end{quotation}
If $\mathsf{P}^+$ is replaced by $\mathsf{P}$ then $\Phi$ is equivalent to a universal Horn sentence where every Horn clause has precisely one non-negated disjunct; these are usually written as implications and called \emph{quasiequations}.
As we explain in Subsection~\ref{subsec:prelimvarieties}, there is almost no difference between a $\mathsf{SP}^+$ class and its $\mathsf{SP}$-closure: they differ by at most one model, up to isomorphism.

The following preservation theorem due to Garrett Birkhoff~\cite{bir} is formulated in the algebraic setting only, though with appropriate definition of ``variety'' and ``equations'', it also holds for general model theoretic structures.
\medskip

\noindent{\bf Preservation Theorem for $\mathsf{HSP}$-classes in algebraic signatures (Birkhoff's Preservation Theorem).}
\begin{quotation}
If $\Phi$ is a first order sentence defining a variety $\mathscr{V}$, then $\Phi$ is logically equivalent to a universally quantified conjunction of equations.
\end{quotation}
A failure of {\L}os-Tarski theorem at the finite level was discovered first by Tait~\cite{tai} and then independently by Gurevich and Shelah; see \cite{alegur} for several variants of the Gurevich-Shelah example.  On the other hand, Rossman~\cite{ros} solved Problem~2 of Alechina and Gurevich \cite{alegur} positively by showing that the Homomorphism Preservation Theorem holds at the finite level, in relational signatures.  

The $\mathsf{SP}^+$\!- and $\mathsf{HSP}$-Preservation Theorems need minor adjustment to be formulated at the finite level because general direct products produce infinite objects from finite ones.  
Thus we restrict to closure under finitary direct products only: $\mathsf{P}_{\rm fin}$.  It turns out that there are many familiar failures of the $\mathsf{HSP}_{\rm fin}$-preservation theorem at the finite level (see Example~\ref{eg:Jtrivial} below for instance).  

For $\mathsf{SP}_{\rm fin}^+$\!-closed classes the situation is far less clear, and the possibility of the corresponding preservation theorem holding at the finite level is Problem 1 of Alechina and Gurevich~\cite{alegur}.   Clark, Davey, Jackson, Pitkethly \cite[Theorem~4.2]{CDJP} implicitly resolved this negatively in algebraic signatures (formally, it is stated for Boolean topological algebras, but the example also works for finite algebras).  
In the present article we make further amendments to this example to provide a simultaneous failure of all three of the $\mathsf{S}$, $\mathsf{SP}_{\rm fin}^+$ and $\mathsf{HSP}_{\rm fin}$-Preservation Theorems at the finite level: the class will be closed under $\mathsf{HSP}_{\rm fin}$, definable by a $\forall^*\exists^*\forall^*$ sentence, but not definable by any first order~$\forall^*\exists^*$ sentence.
Problem 1 from~\cite{alegur} remains tantalisingly open for relational signatures, perhaps the most significant classical model theoretic result for which no finite level resolution is at hand.  The Homomorphism Preservation Theorem presents a similar challenge in the case of algebraic signatures, though the first author has shown that it fails when relativised to the class of bounded lattices~\cite{ham}.

There are many other preservation theorems that we do not address in the present article.  One that does not appear in \cite{alegur,rosen}, but which has recently arisen in the context of variety-like classes of semigroups is the preservation theorem for quantified conjunctions of equations (``equation systems''); see Higgins and Jackson~\cite{higjac}.  
These admit a preservation theorem in terms of closure under taking of elementary embeddings~($\textsf{E}$), homomorphic images and direct products (originally asserted by Keisler~\cite{kei}); see also  \cite{higjac23} for the existential case (which includes closure under the taking of extensions).  We do not know if our example violates the corresponding preservation theorems at the finite level.

\subsection{Pseudovarieties} 
A substantial impetus for the development of finite model theory has been from computer science, via descriptive complexity, and in the context of database theory.  Relational structures have the prominent role in these settings.
An important context in which classes of finite \emph{algebraic} systems have prominence is in the interaction between finite semigroups and regular languages.  
Here, the famous Eilenberg-Sch\"utzenberger correspondence \cite{eil} leads to the study of \emph{pseudovarieties}: classes of finite algebras closed under homomorphisms, subalgebras and finitary direct products.  
The Eilenberg-Sch\"utzenberger correspondence provides an isomorphism between the lattice of semigroup pseudovarieties and the lattice of so-called varieties of regular languages.  This also provides a further motivation for  problems such as $\varmem(*,{\bf B})$, as membership of languages in natural classes of regular languages can be recast as membership problems of their syntactic semigroup in natural pseudovarieties.  

The standard approach to axiomatising pseudovarieties is by way of equations between pseudowords (or pseudoterms in the case of general algebraic systems), which are  equalities between elements of free profinite algebras; a good reference is Almeida's book~\cite{alm} (see Example \ref{eg:Jtrivial} below).   
Our counterexample to the first order ES Problem has no finite axiomatisation by pseudoequations, despite being finitely axiomatised in first order logic (without recourse to pseudoterms); a semigroup example awaits discovery, or a nonexistence proof.

\subsection{Historical note}
Much of the early development of the present article occurred prior to the first author's article \cite{ham},  emerging out of the second author's work in the 2000's~\cite{CDJP} and~\cite{jac:flat} and incorporating project work by the first author under the support of the Australian Mathematical Sciences Institute during the summer of 2010/2011.  
Further development occurred within the second author's funded project \emph{Algebra in Complexity and Complexity in Algebra} (Australian Research Council DP1094578).  
A number of the ideas were also presented by the second author in a talk at the ALCFest conference at Charles University Prague in 2014, and at the Workshop on Finite Model Theory and Multi-Valued Logics, University of Queensland in 2022.

\section{Preliminaries} 
\subsection{Varieties and other algebraic concepts}\label{subsec:prelimvarieties}
We direct the reader to a typical universal algebra text such as \cite{ber} or \cite{bursan} for deeper explanation around the following concepts.  A reader familiar with the area can skip the section entirely.

The algebraic focus of the article is primarily concerned with theory of varieties and to a lesser extent quasivarieties and universal Horn classes, and their restriction to finite algebras.
Recall from the introduction that a variety of algebras is any class $\mathscr{V}$ of algebras of the same signature that is defined by some set (possibly infinite) of equations, or equivalently by Birkhoff's preservation theorem if it is closed under the class operators $\mathsf{H}$, $\mathsf{S}$ and $\mathsf{P}$.  
 We now give some basic examples of classes encountered in the paper.
 \begin{example}
The variety of semilattices is the class of algebras with a single binary operation $\wedge$ defined by idempotence, commutativity and associativity\up:
 \[
 x\wedge x\approx x,\quad x\wedge y\approx y\wedge x, \quad x\wedge (y\wedge z)\approx (x\wedge y)\wedge z,
 \]
 each universally quantified.  
 Equivalently, the variety of semilattices is equal to $\mathsf{HSP}(\mathbf{S}_\wedge)$, where $\mathbf{S}_\wedge$ denotes the $2$-element semilattice \up(which can be represented by interpreting $\wedge$ as the usual number-theoretic multiplication on the two-element set~$\{0,1\}$\up).
 \end{example}
 We will say that an algebra is \emph{semilattice-based} if its fundamental operations include a binary operation satisfying the semilattice axioms.
 
  When ${\bf A}$ is finite, then the finite members of $\mathsf{HSP}({\bf A})$ can be obtained by taking only finitary direct products $\mathsf{P}_{\rm fin}$; in other words we need only consider the pseudovariety $\mathsf{HSP}_{\rm fin}({\bf A})$. 
While the equational theory of ${\bf A}$ continues to define the finitely generated pseudovariety $\mathsf{HSP}_{\rm fin}({\bf A})$ amongst finite algebras (but possibly requiring infinitely many equations), pseudovarieties cannot in general be defined by equations.  
Instead, pseudoequations are the most common approach: equalities between limits of sequences of terms rather than between terms (see Almeida \cite{alm} for example).
  \begin{example}\label{eg:Jtrivial}
  The pseudovariety of $\mathscr{J}$-trivial monoids is the class of finite monoids satisfying 
  \[
\forall a\forall b\forall x_1\forall y_1\forall x_2\forall y_2\ x_1ay_1\approx b\And x_2by_2\approx a\rightarrow a\approx b.
\]
and can also be axiomatised by pseudoequations $x^\omega\approx x^{\omega+1}$ and $(xy)^\omega\approx (yx)^\omega$, where $x^\omega$ denotes the implicit operation defined as the limit $\lim_{n\to\infty}x^{n!}$, the limit of an eventually constant sequence in any finite monoid.
  \end{example}
We will write $\Mod_{\rm fin}(\Sigma)$ to denote the \emph{finite} models of a sentence, or set of sentences $\Sigma$ in some logical framework, and say that a class $\mathscr{K}$ of finite algebras is \emph{first order definable} (amongst finite models) if there is a first order sentence $\Sigma$ such that $\mathscr{K}=\Mod_{\rm fin}(\Sigma)$.
From Example~\ref{eg:Jtrivial}, the $\mathscr{J}$-trivial monoids are first order definable, as $\Mod_{\rm fin}(\Sigma)$, where $\Sigma$ is the finite set consisting of the usual equational axioms for monoids along with the implication displayed in the example.

As described in Subsection \ref{subsec:preservation}, a \emph{universal Horn class} is a class of structures (algebras, relational structures, or general structures with both operations and relations), defined by universally quantified conjunctions of Horn clauses.
These may also equivalently be described as classes closed under taking isomorphic copies of induced substructures ($\mathsf{S}$ again), direct products over nonempty sets of models $\mathsf{P}^+$, and ultraproducts $\mathsf{P}_u$.
\begin{example}
Simple graphs may be considered as symmetric loopless digraphs, and the class of all such structures is a universal Horn class defined by 
\[
\neg (x,x)\in R,\qquad \neg (x,y)\in R\vee (y,x)\in R.
\]
This universal Horn class has no finite generator.
\end{example}
As usual, and as noted in Subsection~\ref{subsec:preservation}, when a Horn clause has exactly one positive disjunct, then we can rewrite it as an implication: so $\neg (x,y)\in R\vee (y,x)\in R$ becomes $(x,y)\in R\rightarrow (y,x)\in R$; the implication in Example \ref{eg:Jtrivial} is another example.  
Sentences of this form are usually called quasiequations, and classes definable by quasiequations are known as quasivarieties.
The distinction between quasivarieties and the more general universal Horn classes is minimal, as they differ only in that for quasivarieties, the direct product closure condition is extended to include the direct product over the empty set of structures.
This degenerate case of direct product returns the one element total structure (any relations in the signature are total on this one point).  
This one-element structure is the only difference between a quasivariety and a universal Horn class, so from a computational and axiomatisability perspective they are not significantly different.

Recall that a class is \emph{locally finite} if for every $n$, the $n$-generated substructures are finite.  
For classes of finite objects this notion is degenerate, but the next notion is not.
A class is \emph{uniformly locally finite} if for all $n$, there exists $m$ such that the $n$-generated substructures have size at most $m$.  
In most cases, what we really use is the less well-known concept \emph{regularly locally finite}, which means that there are only finitely many $n$-generated substructures up to isomorphism, all finite.  In finite signatures however, uniform and regular local finiteness are equivalent, see Bezhanishvili~\cite{bez} (noting that we have bypassed consideration of whether the class in question is hereditary by including ``substructures'' in the definition).
All classes of relational structures are uniformly locally finite, as there is no generation power.  
The following lemma is folklore and we omit the easy proof.
\begin{lem}
Let $\mathscr{K}$ be a uniformly locally finite class of algebras of the same finite signature.  
Then the pseudovariety generated by $\mathscr{K}$ is also uniformly locally finite.
\end{lem}

The next lemma is usually called Birkhoff's Finite Basis Theorem after Garrett Birkhoff~\cite{bir}; versions of it can be found in any universal algebra text; see Theorem 5.27 of \cite{ber}, Theorem~4.2 of Burris and Sankappanavar \cite{bursan}, or Theorem 7.16 of Freese, McKenzie, McNulty and Taylor \cite{ALV} for example.   We give a slight enhancement to the usual presentation, and for this reason only we sketch the proof, even if it is not more than some observations made along the way of the standard proof.  The key addition is that the replacement rule alone is sufficient in syntactic equational inference; see Section~7.2 of~\cite{ALV} for a thorough treatment of syntactic equational inference.
\begin{BFBT}\label{BFBT}
Let $\mathscr{K}$ be a uniformly locally finite class of algebras in a finite signature, let $k$ be a positive integer, and let $\vec{x}$ abbreviate a $k$-tuple of variables $x_0,\dots,x_{k-1}$.  Then  either $\mathscr{K}$ is trivial and satisfies $x\approx y$ or there is a finite set of terms 
$
T_{\rm BFB}\coloneqq \{u_i(\vec{x})\mid i=1,\dots,\ell\}
$
with the following properties\up:
 \begin{enumerate}
 \item $T_{\rm BFB}$ is closed under taking subterms\up;
 \item for every fundamental operation $f$ of arity $n$, and every $i_1,\dots,i_n\leq 
 \ell$ there is precisely one $i'\leq\ell$ for which 
 \[
 \mathscr{K}\models f(u_{i_1}(\vec{x}),\dots,u_{i_n}(\vec{x}))\approx u_{i'}(\vec{x})\up;
 \]
\up(let the set of these satisfied equations be denoted $\Sigma_{\rm BFB}$\up;\up)
  \item for every term $t(\vec{x})$ there is a unique $i\leq \ell$ such that $\mathscr{K}\models t(\vec{x})\approx u_{i}(\vec{x})$ and 
  \item whenever $\mathscr{K}\models t(\vec{x})\approx u_{i}(\vec{x})$ there is a proof of $t(\vec{x})\approx u_{i}(\vec{x})$ from $\Sigma_{\rm BFB}$ that proceeds by a sequence of exact replacements \up(no variable substitutions are required\up).
 \end{enumerate}
 In particular, $\Sigma_{\rm BFB}$ is a finite basis for the $k$-variable equational theory of $\mathscr{K}$.
\end{BFBT}
\begin{proof} 
Because $\mathscr{K}$ is uniformly locally finite, the $k$-generated free algebra ${\bf F}$ for $\mathsf{HSP}(\mathscr{K})$ is finite.  The properties described are then just a description of a canonical presentation of ${\bf F}$ by its operation tables, with item (4) corresponding to the evaluation of $t(\vec{x})$ in the algebra ${\bf F}$ at the free generators.  
For more detail, inductively construct $T_{\rm BFB}$ as a list of representatives of the elements of ${\bf F}$, starting with the variable symbols $x_0,\dots,x_{k-1}$ as free generators, which are distinct elements of~${\bf F}$ and the base case of the generation of $T_{\rm BFB}$.  
Now iterate the following process; for each operation $f$ (of arity $n$, say), and each selection of $n$ elements from the list $u_{i_1}(\vec{x}),\dots,u_{i_n}(\vec{x})$ of so-far constructed terms, check to see if $f(u_{i_1}(\vec{x}),\dots,u_{i_n}(\vec{x}))$ represents a new element of ${\bf F}$: if so, add it to the list; otherwise ignore it.  
We have $T_{\rm BFB}$ at the inevitable termination of this process (${\bf F}$ is finite): when no new elements are added to the list, for all operations and all selections of inputs to those operations from the list so far constructed. 
Properties (1), (2) and (3) follow very easily from this construction.    
For item (4) observe that when an arbitrary term $t(x_0,\dots,x_{k-1})$ is evaluated in ${\bf F}$ at the free generators $x_0,\dots,x_{k-1}$, the evaluation process constitutes a reduction to an element in $T_{\rm BFB}$.  More precisely, we may proceed as follows: take a smallest subterm of $t(x_0,\dots,x_{k-1})$ that does not appear in $T_{\rm BFB}$.  
Because the variables are in $T_{\rm BFB}$, this subterm must be of the form $f(u_{i_1}(\vec{x}),\dots,u_{i_n}(\vec{x}))$ for some fundamental operation $f$ and some terms 
$u_{i_1}(\vec{x})$, \dots, $u_{i_n}(\vec{x})$ in $T_{\rm BFB}$.  
By (2) there is a unique $i'\leq \ell$ such that $\mathscr{K}\models f(u_{i_1}(\vec{x}),\dots,u_{i_n}(\vec{x}))\approx u_{i'}(x_0,\dots,x_{k-1})$ and this law is in $\Sigma_{\rm BFB}$.  
Replace the subterm $f(u_{i_1}(\vec{x}),\dots,u_{i_n}(\vec{x}))$ by the new term $u_{i'}(x_0,\dots,x_{k-1})$.  
Repeating this evaluation strategy eventually leads to an element of $T_{\rm BFB}$.  
\end{proof}

Of course, $T_{\rm BFB}$ and $\Sigma_{\rm BFB}$ depend on both $\mathscr{K}$ and $k$, and we write $T_{\rm BFB}(\mathscr{K},k)$ and $\Sigma_{\rm BFB}(\mathscr{K},k)$ when we wish to make this clear.

\subsection{First order reductions}
We now recall the definition of a first-order reduction; see \cite{imm} or \cite{lib} for a more complete treatment.  
In the standard definition for computational complexity we are allowed recourse to a strict linear order $<$ and work entirely within relational signatures (and typically only consider inputs of at least two elements).  Every operation of arity $n$ can be considered as a relation of arity $n+1$ using its graph.

For $n\in {\mathbb N}\cup\{0\}$, purely relational signatures $\mathcal R$ and $\mathcal S$ and variables $\vec{x}=x_0,\dots, x_{n-1}$, a first-order ${\mathcal R \cup \{<\}}$-formula $\varphi(\vec{x})$ determines a 
family of $n$-ary relations on any ${\mathcal R \cup \{<\}}$-structure $\mathbb B$ in the following way. First, the solution set of~$\varphi(\vec{x})$ is an $n$-ary relation.  Second, for any
fundamental relation $s\in \mathcal S$ of arity $k$, a $kn$-ary formula  $\psi_s(\vec{x}_1,\dots,\vec{x}_k)$
(where $\vec{x}_i$ denotes $x_{i,0}, \dots , x_{i,n-1}$) in the language of ${\mathcal R \cup \{<\}}$ determines a $k$-ary relation on $n$-tuples: a $k$-ary relation on $B^n$. 
The family of formul{\ae}
\[
\{\varphi(\vec{x})\}\cup\{\psi_s(\vec{x}_1,\dots,\vec{x}_k)\mid s\in\mathcal S\}
\]
defines an ${\mathcal S}$-structure $\mu(\mathbb B)$, whose universe $U$ is the solution set of $\varphi(\vec{x})$, as a subset of $B^n$, and where the relations on $\mu(\mathbb B)$ are the restriction to $U$ of the solution sets of $\psi_{s}(\vec{x_1},\dots,\vec{x_k})$, for each $s\in\mathcal S$. These formul{\ae} form an $n$-ary \emph{first order query}.

We will also explore a one-to-many variation of this idea.  In this formulation we have a first order query, as just defined but allow some additional free variables $y_1,\dots,y_k$ in all formul{\ae}.  Each evaluation of these variables in some structure $\mathbb{A}$, say $(y_1,\dots,y_k)\mapsto(a_1,\dots,a_k)$, will give rise to a first order query with parameters $a_1,\dots,a_k$.  The intention is that we have a reduction from problem $\mathcal{P}$ to $\mathcal{Q}$, where structure $\mathbb{A}$ has a  property $\mathcal{P}$ if and only if  every evaluation $(y_1,\dots,y_k)\mapsto (a_1,\dots,a_k)$ in the universe $A$ has property $\mathcal{Q}$.  In particular, we observe that if property $\mathcal{Q}$ is first order, then so also is property $\mathcal{P}$, as we are simply asserting that the structure defined by the first order query has first order property~$\mathcal{Q}$ for all choices of $a_1,\dots,a_k$.

\section{Toolkit: principal congruences and subdirect irreducibility}\label{sec:si}
In this first toolkit section we consider the precise complexity of several computational problems relating to the subdirect decomposition of algebras.  Previous contributions due to Bergman and Slutzki \cite{berslu00,berslu00b,berslu02} give much of the general picture, though we solve one open problem in their work, and find improved complexity for a range of important classes of interest.
The notion of definable completely meet irreducible congruences (introduced in Definition~\ref{defn:cmi}) and their subsequent properties are central to results in Section~\ref{sec:FO} and~\ref{sec:comp}.

Recall that an algebra is \emph{subdirectly irreducible} if there is a pair of distinct elements $a\neq b$ that is contained in the principal congruence of all distinct pairs.  
Equivalently, there is a unique minimum nontrivial congruence: the \emph{monolith}.  
As every algebra is a subdirect product of its subdirectly irreducible quotients,  we obtain a reduction from $\varmem(*,{\bf B})$ to the problem $\varsimem(*,{\bf B})$.  
This reduction can be performed in polynomial time \cite{DDK}, though the authors are unaware of any finer level classification of the complexity of this process.  
We explore related ideas, presenting one specific case where the subdirectly irreducible quotients may be created as first order queries.

We will be interested in the following computational problems on finite algebras.

\fbox{\parbox{0.9\textwidth}{\noindent\textsf{Subdirect irreducibility}\\
Instance: a finite algebra ${\bf A}$ of finite signature.\\
Question: is ${\bf A}$ subdirectly irreducible?}}

\fbox{\parbox{0.9\textwidth}{\noindent\textsf{Principal congruence membership}\\
Instance: a finite algebra ${\bf A}$ of finite signature and two unordered pairs $\{a,b\}$ and $\{c,d\}$ taken from $A$.\\
Question: is $(a,b)$ in the principal congruence of ${\bf A}$ generated by $(c,d)$?}}
\medskip

The principal congruence generated by a pair $(c,d)$ is often denoted $\Cg(c,d)$, so that the question in \textsf{principal congruence membership} can be written ``is $(a,b)\in \Cg(c,d)$?''

\fbox{\parbox{0.9\textwidth}{\noindent\congclass\\
Instance: A pair $({\bf A},C)$, where ${\bf A}$ is an algebra and $C\subseteq A$ is nonempty.\\
Question: Is there a congruence $\theta$ on ${\bf A}$ in which $C$ is a congruence class?}}
\medskip

Before we state our results we recall the definition of a (directed) graph algebra; see Kelarev~\cite{kel} for example.
\begin{defn}
Let $\mathbb{G}=\langle G,E\rangle$ be a directed graph; in other words $E$ is a binary relation on the set $G$.  The \emph{graph algebra} of $\mathbb{G}$ is the algebra on the disjoint union $G\cup\{0\}$ with a single binary operation $\cdot$ defined by 
\[
u\cdot v=\begin{cases} v&\text{ if $(u,v)\in E$}\\
0&\text{ otherwise.}
\end{cases}
\]
\end{defn}
The following theorem is mostly due to Bergman and Slutzki, though item 3 solves positively the problem stated at \cite[p.598]{berslu02}, and therefore negatively the problem stated at \cite[p.603]{berslu02}. 
\begin{thm}\label{thm:sipcm}
In any finite algebraic signature containing an operation of arity at least $2$\up:
\begin{enumerate}
\item \up(\cite{berslu02}\up) \textsf{Principal congruence membership} is \texttt{NL}-complete.
\item \up(\cite{berslu02}\up) \textsf{Subdirect irreducibility} is \texttt{NL}-complete.
\item \congclass\ is \texttt{NL}-complete.
\end{enumerate}
\end{thm}
\begin{proof}
All three statements are proved by reduction from the directed graph reachability problem.  Items 1 and 2 are \cite[Theorem 3.5]{berslu02} (which also considers the problem of recognising simplicity for a finite algebra; we have no use for this here though).  
We consider item 3.  
It is already shown in \cite{berslu02} that \congclass\ is in  \texttt{NL}.  
We need to show hardness, reducing from the \texttt{NL}-complete problem of st-connectivity, also known as directed graph reachability, and noting that $\text{co-}\texttt{NL}=\texttt{NL}$ by the Immerman–Szelepcs\'{e}nyi theorem~\cite{imm}.

Consider a finite directed graph $\mathbb{G}=\langle V;E\rangle$ with distinguished vertices $u,v$.  
Let $a,b,0$ be symbols not in $V$ and construct the graph $\mathbb{G}_{a,b}$ on $V\cup\{a,b\}$ with edges consisting of $E$ along with $(a,u)$, $(v,a)$ and $(v,b)$.  
We claim that in the graph algebra of $\mathbb{G}_{a,b}$, the subset $\{a,b\}$ is a congruence class if and only if $v$ is not reachable from $u$.  Clearly $\{a,b\}$ is a congruence class if and only if it is the congruence class of the congruence $\Cg(a,b)$, thus it suffices to show that $\{a,b\}$ is a congruence class of $\Cg(a,b)$ if and only if $v$ is not reachable from $u$.

Assume $v$ is reachable from $u$.  So $v$ is also reachable from $a$, and then $b$ is reachable from $a$; let $a=a_0,u=a_1,a_2,\dots,a_n=v$ be a sequence of vertices leading from $a$ to $v$.  So $(\dots ((aa_1)a_2)\dots a_n)b=vb=b$ and then $a\equiv b$ implies $b\equiv (\dots ((ba_1)a_2)\dots a_n)b=0$.  So $\{a,b\}$ is not a congruence class.  

Now assume that $v$ is not reachable from $u$.  
Let $U$ denote the subset of $V\cup\{a,b\}$ reachable by a (possibly $0$-length) path from $u$, and observe that $a,b,v\notin U$.  
It is now routine to show that $\Cg(a,b)$ consists of the following blocks: $\{a,b\}$, $U\cup\{0\}$ and all singletons from $V\backslash U$; in particular $(u,v)\notin \Cg(a,b)$.

This proof concerns only the case where there is a single binary operation, but the argument can be transformed trivially into any signature containing at least one operation $f$ of arity $\geq 2$ by defining $f$ by $f(a_1,a_2,\dots)\coloneqq a_1\cdot a_2$, where $\cdot$ is the operation, and letting all other operations be projections (and treating nullaries as arbitrary constants).
\end{proof}

As we shortly demonstrate, a wide range of algebraic settings enjoy a weak congruence condition that improves on Theorem \ref{thm:sipcm}.  Let $T_x$ be the set of all terms in an algebraic signature in variables $\{x,z_1,z_2,\dots\}$, and in which the variable $x$ explicitly occurs, and let $F$ be a subset of $T_x$.  For an algebra ${\bf A}$ and $a,b,c,d\in A$ we write $\{c,d\}\collapse^1_F\{a,b\}$ if $a=b$ or there is a term $t(x,z_1,z_2,\dots) \in F$ and elements $e_0,e_1,\dots\in A$ such that 
\[
\{t^{\bf A}(c,e_0,e_1,\dots),t^{\bf A}(d,e_0,e_1,\dots)\}=\{a,b\}.
\]  
Define  $\{c,d\}\collapse^{n}_F\{a,b\}$ if there is a sequence $a=a_0,a_1,a_2,\dots,a_n=b$ with $\{c,d\}\collapse_F^1\{a_i,a_{i+1}\}$ for each $i=0,\dots,n-1$.  Finally, write $\{c,d\}\collapse_F\{a,b\}$ if $\{c,d\}\collapse_F^k\{a,b\}$ for some $k$.  
(This notation is based on Baker, McNulty and Wang~\cite{BMW}, see also the book \cite[p.~25]{ALV}.)  
Maltsev proved that a pair $(a,b)\in \Cg(c,d)$ if and only if $\{c,d\}\collapse_{T_x}\{a,b\}$.  
A class has \emph{term finite principal congruences} (TFPC)~\cite{CDFJ}  if there is a finite set $F\subseteq T_x$ such that $\collapse_{T_x}$ coincides with $\collapse_{F}$.  
A class has \emph{definable principal congruences} (DPC) if there exists $n$ and a finite subset $F\subseteq T_x$ with ${\collapse_{T_x}}={\collapse_F^n}$.  
The DPC property is rather special, but as explained in~\cite{CDFJ,DJMM,wan}, many familiar classes of algebraic structures have the~TFPC property, including semigroups, groups, rings and finitely generated congruence distributive varieties.  

The following generalised version of \textsf{principal congruence membership} is useful before we state and prove the theorem.

\fbox{\parbox{0.9\textwidth}{\noindent\textsf{Congruence membership}\\
Instance: a finite algebra ${\bf A}$ of finite signature, a nonempty subset $S\subseteq A^2$ and an unordered pair $\{a,b\}$.\\
Question: is $(a,b)$ in the  congruence of ${\bf A}$ generated by $S$?}}
\medskip

\begin{thm}\label{thm:TFPC}
Let $\mathscr{V}$ be a variety of finite signature and with TFPC.  
\begin{enumerate}
\item[(0)] For any finite algebra ${\bf A}\in\mathscr{V}$ the \textsf{congruence membership} problem is solvable in~\texttt{L}.
\end{enumerate}  
As a consequence, each of the following problems for $\mathscr{V}$ are in \texttt{L}\up:
\begin{enumerate}
\item \textsf{Principal congruence membership}\up;
\item \textsf{Subdirect irreducibility}\up;
\item \congclass.
\end{enumerate}
\end{thm}
\begin{proof}[Proof of Theorem \ref{thm:TFPC}.]
We first explain why the solvability of (1), (2) and (3) in~\texttt{L} follows from the logspace solvability of (0) almost trivially: in each case we can solve in logspace using an oracle for (0), and as $\texttt{L}^\texttt{L}=\texttt{L}$ we obtain membership in~$\texttt{L}$.  Note that space considerations still only concern the working tape, and not the oracle tape, so that we do not need to concern ourselves with the size of any constructed $S$ written to the oracle, only with the working tape required to perform any calculations needed for that writing; see \cite[Section~6.2]{BDJN} for detailed discussion on this.

Item (1) is a specific case of (0), where $S=\{(c,d)\}$.  

For item (2), subdirect irreducibility can be decided in \texttt{L} by  systematically enumerating (one-by-one, re-using the same piece of work tape) all pairs $a\neq b$, and in each case systematically enumerating (one-by-one, reusing some new piece of work tape) all $c\neq d$, testing for satisfaction of $(a,b)\in\Cg(c,d)$ using the oracle.  Subdirect irreducibility holds if there is a choice of $(a,b)$ such that every pair $c\neq d$ satisfies $(a,b)\in\Cg(c,d)$.

For item (3), observe that a nonempty subset $C\subseteq A$ is a congruence class of a congruence if and only if it is a congruence class of the congruence $\Cg(S)$, where $S=C^2\subseteq A^2$.  
Because $C$ is trivially a subset of a congruence class of $\Cg(C^2)$, we can decide equality with a congruence class, by selecting an arbitrary $c\in C$, and then systematically verifying $(d,c)\notin \Cg(C^2)$ using the oracle for (0), for each $d\notin C$.  
Again we have reduced to a special case of (0), where $S=C^2$.

To prove (0), let $F$ be a finite subset of $T_x$ that determines principal congruences in $\mathscr{V}$, let ${\bf A}$ be a finite member of $\mathscr{V}$ and $S\subseteq A^2$.  
For any given pair $(a,b)\in A^2$, if $a=b$ then $(a,b)\in \Cg(S)$ trivially, so we assume that $a\neq b$.  Then $(a,b)\in\Cg(S)$ if and only if $a$ is reachable from $b$ in the undirected graph on $A$ whose undirected edges are $\{\{u,v\}\in A^2\mid S\collapse^1_F\{u,v\}\}$.  
As $F$ is finite, the relation $\{c,d\}\collapse^1_F\{u,v\}$ is a first order formula for each $c\neq d$ with $(c,d)\in S$.  
The relation $S\collapse^1_F\{u,v\}$ is simply the disjunction of the formul{\ae} $\{c,d\}\collapse_F^1\{u,v\}$ over all pairs $(c,d)\in S$, and hence is also a first order formula, or uniformly: $\exists c\exists d\, \big(S(c,d)\And \{c,d\}\collapse_F^1\{u,v\}\big)$.
Thus the edge relation of this graph is itself first order definable from the input subset $S$ and algebra ${\bf A}$, and so can be constructed in logspace.  
Thus (0) holds because undirected graph reachability can be solved in~\texttt{L}; see \cite{rei,rei2}.
\end{proof}

Theorem \ref{thm:TFPC} (2) shows that, under the TFPC assumption,  the promise of subdirect irreducibility in restricting $\varmem(*,{\bf B})$ to $\varsimem(*,{\bf B})$ can be verified in logspace.  
A more useful contribution would be if it gave rise to a logspace reduction from $\varmem(*,{\bf B})$ to $\varsimem(*,{\bf B})$, however it is not clear to the authors if this is true; even in nondeterministic logspace.  
\begin{question}\label{question:logspacesi}
Is there a reasonably applicable algebraic condition that guarantees the logspace construction of subdirectly irreducible quotients of a finite input algebra ${\bf B}$? 
\end{question}
It would be interesting if the TFPC condition was sufficient for example or some well known Maltsev condition.  
We will give two example solutions to Question~\ref{question:logspacesi}: in Proposition~\ref{pro:divisionordered} and in Lemma~\ref{lem:si}.  
These are strong solutions in the sense that the subdirectly irreducible quotients may be constructed as a first order query (and hence within deterministic logspace), thereby showing that $\varmem(*,{\bf B})$\ and $\varsimem(*,{\bf B})$ are in fact first order equivalent.  
Neither are ``classical'' universal algebraic properties however.
\begin{defn}\label{defn:cmi}
Let $\mathscr{K}$ be a class of  algebras of the same signature.  
The class $\mathscr{K}$ is said to have \emph{definable completely meet irreducible congruences} \up(abbreviated to \emph{definable cmi congruences}\up) if there is a formula $\pi_{x,y}(u,v)$ with four free variables $x,y,u,v$ such that for every ${\bf A}\in \mathscr{K}$ and every $a,b\in A$ with $a\neq b$, the binary relation $\{(c,d)\in A^2\mid {\bf A}\models \pi_{a,b}(c,d)\}$ defined by the formula $\pi_{a,b}(u,v)$ is a congruence that is maximal with respect to not containing the pair $(a,b)$.
\end{defn}
We remind the reader that the property that a congruence $\theta$ is maximal with respect to $(a,b)\notin\theta$ for some $a,b$ is equivalent to $\theta$ being completely meet irreducible in the congruence lattice.  When $K$ is a class of finite algebras, then the congruence lattice is finite, so complete meet irreducibility coincides with meet irreducibility, and so for the purposes of this paper, either notion would suffice.

In general there may not be a unique maximum congruence separating a pair $a\neq b$, but Definition \ref{defn:cmi} only requires that $\pi_{a,b}$ be a maximal congruence separating $a$ from $b$.

The following gives a one-to-many first order reduction from an algebra with definable cmi congruences to its subdirectly irreducible quotients.  This in turn provides a one-to-many first order reduction from $\varmem(*,{\bf B})$ to $\varsimem(*,{\bf B})$.
\begin{thm}\label{thm:fodecomp}
If $\mathscr{K}$ is a class of algebras of the same signature, closed under taking homomorphic images and with definable cmi congruences, then there is a first order query $\phi(x,y)$ with two free variables $x,y$, such that for every finite ${\bf A}\in \mathscr{K}$ and every $a\neq b$ in $A$ we have $\phi(a,b)$ defines a subdirectly irreducible quotient of ${\bf A}$ in which the monolith  is generated by representatives of the congruence classes of $a$ and $b$.
\end{thm}
\begin{proof}
Let $\pi_{x,y}(u,v)$ be a formula defining cmi congruences in $K$, and for each ${\bf A}\in K$ and each $a\neq b$ in $A$, let $\pi_{a,b}$ denote the cmi congruence defined by $\pi_{a,b}(u,v)$.  Our one-to-many reduction will  use an implicit linear order $\leq$ on the set $A$ (with $<$ being the corresponding strict order).  The proof also works for infinite algebras, if $\leq$ is an implicit well-ordering. 

We first observe that we may use $\leq$ to define the set of minimal representatives of the $\pi_{a,b}$-classes: $\{z\in A\mid \forall x (\pi_{a,b}(x,z)\Rightarrow z\leq x)\}$.  On this set we define each fundamental $n$-ary operation $f$ as follows:
\[
(f(c_1,\dots,c_n)\approx d)\leftrightarrow (\pi_{a,b}(f(c_1,\dots,c_n),d)).
\]
The operations are well defined because $\pi_{a,b}$ is a congruence relation, and as we have selected the domain as a transversal of the congruence classes, the defined algebra is a subdirectly irreducible quotient of ${\bf A}$ in which the representative members of the $\pi_{a,b}$ congruence classes of $a$ and $b$ generate the minimal congruence.
\end{proof}
Note that when there is a semilattice operation $\wedge$, we may avoid recourse to the implicit linear order $\leq$ by instead taking $\wedge$-minimum of each congruence class.  We deploy this approach in Section~\ref{sec:FO}.
\begin{cor}
If $\mathscr{V}$ has definable cmi congruences, by formula $\pi_{x,y}(u,v)$, then~$\mathscr{V}_{\rm si}$ is definable within $\mathscr{V}$ by the sentence 
\[
\exists x\exists y\forall u\forall v\ x\neq y\And (u\approx v\leftrightarrow \pi_{x,y}(u,v)).
\]
\end{cor}
We mention that subdirectly irreducible algebras are also first order definable in the presence of definable principal congruences, as subdirect irreducibility is equivalent to 
\[
\exists x\exists y \forall u\forall v\  x\not\approx y\And \big(u\not\approx v\rightarrow (x,y)\in \Cg(u,v)\big).
\]

To complete the section we give an example property that guarantees definable cmi congruences, and therefore also the identification of subdirectly irreducible quotients as first order queries.  For any term $t(x,z_1,\dots)$
in $T_x$, we define the formula ${\divides_t}(x,y)$ by 
\[
\exists z_1\exists z_2\dots t(x,z_1,z_2,\dots)\approx y.
\]
If ${\bf A}$ is an algebra and $a,b\in A$ are such that ${\divides_t}(a,b)$ holds on ${\bf A}$, then we say that  $a$ \emph{divides} $b$ by way of $t$, and write $a\divides_t b$.  More generally, for a set of terms $F\subseteq T_x$ we similarly write $\divides_F$ to denote the possibly infinite disjunction  $\bigvee_{t\in F}\divides_t$.  We say that $a$ divides $b$ to mean that $a\divides_tb$ for some $t$, or equivalently if $a\divides_{T_x}b$ (restricting to those terms of $T_x$ explicit in $x$).  A class has \emph{finitely determined division} (FDD) if the relation $\divides$ is first order definable, or equivalently if there is a finite set of terms $F$ such that ${\divides}={\divides_F}$ (this equivalence is at the level of arbitrary algebraic structures, not finite algebraic structures; \cite[Theorem 5.1]{jactro10}).  As is shown in Example~4.3 of \cite{jactro10}, the FDD property properly generalises the TFPC property.

The relation $\divides$ is always a preorder, but if it is an order relation on some algebra (or class of algebras), then we say that algebra (or class of algebras, respectively) is \emph{division ordered}.    
Theorem~7.3 of~\cite{jactro10}  shows that the finite division ordered algebras of a given signature  form a pseudovariety\footnote{Theorem~7.3 of \cite{jactro10} concerns the variety of an individual algebra, but the proof  depends only on showing that a quotient of a locally finite division ordered algebra is division ordered, and this is sufficient to deduce that the class of finite division ordered algebras is closed under taking quotients; it is trivially closed under the taking of finitary direct products and  subalgebras.} and contain a number of finite algebras with negative universal algebraic properties: Lyndon's algebra \cite{lyn} for example.   
The restricted case of division ordered semigroups is particularly well studied and known as the $\mathscr{J}$-trivial semigroups (see Exercise 5.1.5 in Almeida \cite{alm}) and have also been an important source of challenging examples in semigroup varieties; see Gusev and Sapir~\cite{gussap} for one of many examples.    

We have framed the following result in the context of universal Horn classes, because the division order property is defined by quasiequations: $\{x\divides_t y\And y\divides_s x\rightarrow x\approx y\mid s,t\in T_x\}$.   As we shall shortly see (Lemma \ref{lem:ulf}(i) below, due to Gorbunov), the result also applies to pseudovarieties of division ordered finite algebras.  
\begin{pro}\label{pro:divisionordered}
Let $\mathscr{K}$ be a universal Horn class of algebras with TFPC.  If $\mathscr{K}$ is division ordered, then $\mathscr{K}$ has definable cmi congruences.
\end{pro}
\begin{proof}
First note that Example~4.2 of \cite{jactro10} shows that the TFPC condition ensures that there is a finite set of terms $F\subseteq T_x$ such that $\mathscr{K}$ has TFPC and FDD by way of $F$. In particular $\divides$ coincides with $\divides_{F}$ and so is a first order definable relation.  We also note that there is no loss of generality in assuming that $x\in F$; this is a particular case of the General Shadowing Theorem~3.1 of \cite{CDFJ}.
As $\mathscr{K}$ is division ordered, either $a\not\divides b$ or $b\not\divides a$.  Assume that $a\not\divides b$ and define $J_b\coloneqq \{c\mid c\not\divides b\}$, noting that $a\in J_b$ and $b\notin J_b$.  
As $\divides$ is first order definable in $\mathscr{K}$, so also is the unary relation $J_b$.  Now recall the \emph{syntactic congruence} $\sim_{S}$ of a subset $S\subseteq A$, as defined by 
\[
u\sim_{S}v\Leftrightarrow \bigand_{t\in F}\forall \vec{z}\, t(u,\vec{z})\in S\leftrightarrow t(v,\vec{z})\in S.
\]
The TFPC condition with respect to $F$ ensures this is a congruence, and indeed the largest congruence in which $S$ is a union of congruence classes (see \cite[Corollary~2.4]{CDFJ} for the specific usage of TFPC, though the underlying idea of syntactic congruence is very well known in semigroup theoretic contexts and dates back at least as far as the 1954 paper of Pierce~\cite{pie}).  
Now the relation $\sim_S$ is a first order formula in the unary predicate $S$.  
As $J_b$ is a first order definable unary relation, it follows that $\sim_{J_b}$ is a first order definable congruence relation.  
We now claim that it is a maximal congruence separating $a$ from $b$.  
First observe that $J_b$ is a single congruence class of~$\sim_{J_b}$. 
In particular, as $a\in J_b$ we have that $a\not\sim_{J_b}b$.  
We claim that any congruence $\theta$ properly extending $\sim_{J_b}$ collapses $b$ to some element of $J_b$.  
As  $\sim_{J_b}$ is the maximum congruence with respect to not collapsing an element from $A\backslash J_b$ to an element of $J_b$, it follows that there is $b'\in A\backslash J_b$ and $a'\in J_b$ with $a'\mathrel{\theta}b'$.  
By the definition of $J_b$ there is a term $t(x,\vec{z})\in F$, and elements $\vec{c}$ from $A$ such that $t^{\bf A}(b',\vec{c})=b$.  Then $t^{\bf A}(b',\vec{c})$ is congruent to $t^{\bf A}(a',\vec{c})$ modulo $\theta$.  
But $t^{\bf A}(a',\vec{c})\in J_b$ as $a'\in J_b$ and $J_b$ is an absorbing ideal of ${\bf A}$.  
This shows that the first order relation $\sim_{J_b}$ is a maximal congruence with respect to not identifying $a$ and $b$; but we assumed $a\not\divides b$.  
As $\mathscr{K}$ is division ordered, either $a\not\divides b$ or $b\not\divides a$ and so $\mathscr{K}$ has definable cmi congruences by way of the  formula 
\[
\pi_{x,y}(u,v)=\left(x\not\divides_{F}y\And u\sim_{J_y} v\right)\vee\left(x\divides_{F}y\And y\not\divides_{F}x\And u\sim_{J_x} v\right).\qedhere
\]
\end{proof}
We mention the work of Schein~\cite{sch66}, where for example Theorem~3.7 (which concerns semigroups with zero for which $\sim_{\{0\}}$ is the identity relation) or Theorems~4.5 and~4.6 (which concern a subclass of the $\mathscr{J}$-trivial semigroups) might serve as a motivation for the proof idea of Proposition~\ref{pro:divisionordered}, even if \cite{sch66} concerns the structure of subdirectly irreducible semigroups rather than the potential definability of subdirectly irreducible quotients.

\section{Toolkit: Preservation Theorems}\label{sec:preservation}
 In this section we provide some preservation theorems that hold at the finite level in algebraic signatures.  
We use some of these to provide information on tightness to the number of quantifier alternations in the results of Section~\ref{sec:FO}.
 
 In the class of relational structures, $\mathsf{HSP}$-classes are too restrictive to be of much interest.  However there is a rather elementary preservation theorem in universal first order logic that can be stated for $\mathsf{SP}$-classes for arbitrary finite structures (relational or algebraic), provided uniform local finiteness is assumed.  To begin with we recall some basic facts concerning classes generated by classes of finite structures.  
 In the following, $\Th_{\rm uH}(\mathscr{K})$ denotes the universal Horn theory of $\mathscr{K}$.
\begin{lem}\label{lem:ulf}
The following hold for any $\mathsf{SP}_{\rm fin}$-closed class $\mathscr{K}$ of finite structures in a finite signature.
\begin{enumerate}
\item $\mathscr{K}=\Mod_{\rm fin}(\Th_{\rm uH}(\mathscr{K}))$.
\item $\mathscr{K}$ is uniformly locally finite if and only if  $\Mod(\Th_{\rm uH}(\mathscr{K}))$ is locally finite.
\item $\mathscr{K}$ is uniformly locally finite if and only if the ultraproduct closure of $\mathscr{K}$ is locally finite.
\item If $\mathscr{K}$ is a uniformly locally finite class of algebras and is also closed under taking homomorphic images, then $\mathscr{K}$ is the finite part of $\mathsf{HSP}(\mathscr{K})$.
\end{enumerate}
\end{lem}
\begin{proof}
The first part is due to Gorbunov \cite[Proposition 2.5.12]{gor}, while the second part is usually attributed to Maltsev.  
The ``only if'' statement of part (3) follows from part (2) because the ultraproduct closure of $\mathscr{K}$ is a subclass of the universal Horn class generated by $\mathscr{K}$.
For the ``if'' statement, we prove the contrapositive.  
Assume that $\mathscr{K}$ is not uniformly locally finite, so that there exists an $n\in\mathbb{N}$ such that for every $m\in\mathbb{N}$ there is an $n$-generated member ${\bf S}_m$ of $\mathscr{K}$ with more than $m$ elements.
Add constants $a_1,\dots,a_n$ to name the $n$ generators of each ${\bf S}_m$.  
For each $k\in\mathbb{N}$ there are, up to isomorphism, only finitely many $\{a_1,\dots,a_n\}$-generated structures of size up to $k$ and so we may create a variable-free sentence $\sigma_k$ (in the expanded signature including the constants $a_1,\dots,a_n$) that asserts that ${\bf S}_m$ is not amongst them, whenever $m\geq k$.
For example, $\sigma_k$ could be the negation of the disjunction of the diagrams of (representatives of the isomorphism classes) of all at most $k$-element structures generated by $a_1,\dots,a_n$; see \cite[Definition~8.35]{ALV}.
Then the ultraproduct of $\{{\bf S}_m\mid m\in \mathbb{N}\}$, over any nonprincipal ultrafilter on~$\mathbb{N}$, also satisfies $\sigma_k$, and as this is true for every $k$, it follows that the substructure of the ultraproduct generated by $a_1,\dots,a_n$ is infinite.
So, the ultraproduct closure of $\mathscr{K}$ is not locally finite, as required for the contrapositive of the ``if'' statement of (3).

For part (4), note that every variety is generated by its finitely generated free algebras, and all finite algebras are quotients of free algebras.  Uniform local finiteness ensures that finitely generated free algebras in the variety generated by~$\mathscr{K}$ are finite, and as~$\mathscr{K}$ is assumed to be closed under quotients, these free algebras, and their quotients, lie in $\mathscr{K}$.  In other words, $\mathscr{K}$ is the finite part of $\mathsf{HSP}(\mathscr{K})$.
\end{proof}

The following easy observation gives some partial preservation theorems at the finite level.
\begin{thm}\label{thm:Aequation}
Let $\Phi$ be a {universal} sentence defining a locally finite class.
\begin{enumerate}
\item If the finite models $\mathscr{K}=\Mod_{\rm fin}(\Phi)$ are
$\mathsf{SP}_{\rm fin}$-closed then $\Mod(\Phi)$ is a universal Horn class and $\Phi$ is logically equivalent to a universal Horn sentence, even amongst structures of possibly infinite size.
\item \up(For algebraic signatures.\up)  If the finite models $\mathscr{K}=\Mod_{\rm fin}(\Phi)$  are a pseudovariety \up($\mathsf{HSP}_{\rm fin}$-closed\up) then $\Mod(\Phi)$ is a variety and $\Phi$ is logically equivalent to a finite conjunction of equations, even amongst structures of possibly infinite size.
\end{enumerate}
\end{thm}
\begin{proof}
To prove (1), first observe that as $\mathscr{K}\subseteq \Mod(\Phi)$ and $\Mod(\Phi)$ is locally finite and ultra\-product-closed, it follows from Lemma \ref{lem:ulf} part (3) that $\mathscr{K}$ is uniformly locally finite.   
Then Lemma \ref{lem:ulf} part (2) implies that $\Mod(\Th_{\rm uH}(\mathscr{K}))$ is locally finite.  

Next observe that two locally finite universal classes agreeing on finite models coincide, because all models embed into an ultraproduct of their finitely generated substructures.  So $\Mod(\Th_{\rm uH}(\mathscr{K}))=\Mod(\Phi)$, and then the Compactness Theorem  ensures that there is a finite subset of $\Th_{\rm uH}(\mathscr{K})$ that is equivalent to $\Phi$.

The proof of (2) is essentially identical, using Lemma \ref{lem:ulf}(3), but then \ref{lem:ulf}(4) in place of~\ref{lem:ulf}(2).
\end{proof}
The following theorem differs subtly from Theorem~\ref{thm:Aequation} in that there is no requirement that $\Mod(\Phi)$ be locally finite, only that $\Mod_{\rm fin}(\Phi)$ be uniformly locally finite.  
This is a much weaker assumption, and the conclusion is weaker in dropping the conclusion of equivalence on all structures.
We phrase it for algebraic signatures, though it can be adapted for general signatures.
\begin{thm}\label{thm:Aequation2}
Let $\Phi$ be a universal sentence in an algebraic signature whose finite models $\mathscr{K}=\Mod_{\rm fin}(\Phi)$ are $\mathsf{SP}_{\rm fin}^+$-closed and uniformly locally finite.
Then $\Phi$ is equivalent on finite algebras to a universal Horn sentence, and if $\mathscr{K}$ is a pseudovariety, then $\Phi$ is equivalent on finite structures to a finite conjunction of equations.
\end{thm}
\begin{proof}
Let $n$ be the number of variables appearing in $\Phi$.
Birkhoff's Finite Basis Theorem~\ref{BFBT} guarantees that the $n$-variable equational theory of $\mathscr{K}$ has a finite basis $\Sigma_{\rm BFB}$, and all $n$-variable terms reduce modulo $\Sigma_{\rm BFB}$ to one in the finite set $T_{\rm BFB}$.
Up to logical equivalence, there are only finitely many $n$-variable atomic formul{\ae} and only finitely many universal Horn sentences involving only terms from~$T_{\rm BFB}$.
Let $\Sigma$ be the conjunction of $\Sigma_{\rm BFB}$ along with the finitely many satisfied universal Horn sentences involving only terms $T_{\rm BFB}$ (and when $\mathscr{K}$ is a pseudovariety, it will suffice to take only $\Sigma_{\rm BFB}$).
The universal Horn sentence $\Sigma$ determines when an $n$-generated finite algebra lies in $\mathscr{K}$.
But as $\mathscr{K}=\Mod_{\rm fin}(\Phi)$ and $\Phi$ is an $n$-variable universal sentence, membership of a finite algebra ${\bf A}$ in $\mathscr{K}$ is equivalent to membership in $\mathscr{K}$ of the (at most) $n$-generated subalgebras of ${\bf A}$.
Thus $\Sigma$ is equivalent to~$\Phi$ on finite algebras.
\end{proof}

The following result is given for function-free signatures in \cite[Theorem 3]{gur2}, where it is attributed to Compton.  Recall that a \emph{hereditary class} means one that is closed under $\mathsf{S}$.  For an algebra ${\bf A}$ and subset $B\subseteq A$ we let $\sg{\bf A}{B}$ denote the subalgebra of ${\bf A}$ generated by $B$.
\begin{thm}\label{thm:forallexists}
If $\Phi$ is a $\forall^*\exists^*$-sentence defining a uniformly locally finite hereditary class $\mathscr{K}$ of finite structures 
then $\Phi$ is equivalent on finite structures to a universal sentence.
\end{thm}
\begin{proof}
The proof is by way of a sort of ``localised Skolemisation'' process.  Let $\Phi$ be 
\[
\forall x_1\forall x_2\dots\forall x_n\exists y_1\exists y_2\dots \exists y_m \phi(x_1,\dots,x_n,y_1,\dots,y_m),
\]
where $\phi$ is quantifier free.
We first observe that uniform local finiteness allows us to apply Birkhoff's Finite Basis Theorem \ref{BFBT} to find a finite set $F$ of terms in variables $x_1,\dots,x_n$ such that every term $t(x_1,\dots,x_n)$ is equivalent in $\mathscr{K}$ to one from $F$; that is, for every $t(x_1,\dots,x_n)$ there exists $s\in F$ with $\mathscr{K}\models s\approx t$.  

We claim that $\Phi$ is equivalent to the  sentence $\Phi'=\forall x_1\dots\forall x_n\phi'(x_1,\dots,x_n)$ where
\[
\phi'(x_1,\dots,x_n)\coloneqq  \bigvee_{(t_1,\dots,t_m)\in F^m}\phi(x_1,\dots,x_n,t_1(x_1,\dots,x_n),\dots,t_m(x_1,\dots,x_n)).
\]
It is trivial that $\Phi'\vdash \Phi$, so that the finite models of $\Phi'$ are a subclass of $\mathscr{K}$.  Now assume that ${\bf A}\in\mathscr{K}$.  Consider any $a_1,\dots,a_n\in A$.  Now as $\sg{{\bf A}}{a_1,\dots,a_n}\models \Phi$, there are $b_1,\dots,b_m\in \sg{{\bf A}}{a_1,\dots,a_n}$ with $\sg{{\bf A}}{a_1,\dots,a_n}\models \phi(a_1,\dots,a_n,b_1,\dots,b_m)$.  As $b_1,\dots,b_m\in \sg{{\bf A}}{a_1,\dots,a_n}$ there are terms $t_1,\dots,t_m\in F$ with $b_i=t_i^{\bf A}(a_1,\dots,a_n)$ for $i=1,\dots,m$.  Then 
\[
\sg{{\bf A}}{a_1,\dots,a_n}\models \phi(a_1,\dots,a_n,t_1^{\bf A}(a_1,\dots,a_n),\dots,t_m^{\bf A}(a_1,\dots,a_n)),
\]
and hence $\sg{{\bf A}}{a_1,\dots,a_n}\models\phi'(a_1,\dots,a_n)$.  That is, ${\bf A}\models \phi'(a_1,\dots,a_n)$.  As $a_1,\dots,a_n$ were arbitrary, it follows that ${\bf A}\models \Phi'$.
\end{proof}
The following result also holds if ``pseudovariety'' is replaced by ``$\mathsf{SP}_{\rm fin}^+$-closed class'', and ``finite conjunction of equations'' is replaced by ``universal Horn sentence''.
\begin{cor}\label{cor:compton}
If $\Phi$ is a $\forall^*\exists^*$ sentence defining a uniformly locally finite pseudovariety $\mathscr{K}$ of finite algebras, 
then $\Phi$ is equivalent on finite structures to a finite conjunction of equations.
\end{cor}
\begin{proof}
Theorem \ref{thm:forallexists} shows that $\Phi$ is equivalent to a universal sentence, and then Theorem  \ref{thm:Aequation2} shows this is equivalent to a conjunction of equations.
\end{proof}
Another way to state Corollary \ref{cor:compton} is that a uniformly locally finite pseudovariety~$\mathscr{K}$ of finite algebras is either finitely axiomatised by equations, or cannot be defined by any $\forall^*\exists^*$-sentence.

We now observe that the assumption of uniform local finiteness in these statements is necessary (a similar example can be found in \cite{rosen}).
\begin{eg}\label{eg:Jtrivial2}
The class of finite $\mathcal{J}$-trivial monoids is definable by a universal Horn sentence, but not by any system of equations.  So the Preservation Theorem for $\mathsf{HSP}$-classes fails at the finite level, even amongst universal Horn sentences.  
\end{eg}
\begin{proof}
Our definition of $\mathcal{J}$-triviality in Example \ref{eg:Jtrivial} is by way of a universal Horn sentence.  
However as all finite nilpotent monoids (that is, monoids obtained by adjoining an identity element to a nilpotent semigroup) are $\mathcal{J}$-trivial, all equations true of $\mathcal{J}$-trivial monoids are true of all monoids.  
So the class of finite $\mathcal{J}$-trivial monoids is not the set of finite models of any system of equations (and certainly not of any finite conjunction of equations).
\end{proof}
Example~\ref{eg:Jtrivial2}  shows that Theorem \ref{thm:Aequation}(2) fails to hold if the local finiteness condition is dropped, and that Theorem \ref{thm:Aequation2} \up(equational statement\up) and Corollary \ref{cor:compton} fail to hold if the assumption of uniform local finiteness is dropped.
To find failures of the other parts of these results under similarly relaxed assumptions we give a simple ad hoc example.
\begin{eg}\label{eg:noLF2}
In the signature of three unary operations $f,g,h$, the finite models of the universal class defined by $f(g(x))\approx x\And  \left(g(f(x))\approx x\vee h(x)\approx x\right)$ form a pseudovariety, but the arbitrary sized models are not closed under direct products, so are not even a universal Horn class.
\end{eg}
\begin{proof}
On finite models, the property $f(g(x))\approx x$ ensures that $f$ and $g$ are inverse permutations, so that $g(f(x))\approx x$ also holds.  
So the first conjunct implies the second amongst finite models, and the finite models can be defined by the single equational law $f(g(x))\approx x$.  
Now let ${\bf A}$ be the 2-element model of the sentence on $\{0,1\}$ with $f(i)\coloneq g(i)\coloneq  h(i)\coloneq 1-i$ and let ${\bf B}$ be the model of this sentence on $\{0,1,\dots\}$ given by $f(i)\coloneq \max\{0,i-1\}$, $g(i)\coloneq  i+1$ and $h(i)\coloneq i$.  
Then the direct product ${\bf A}\times {\bf B}$ fails the sentence at $x=(0,0)$, because $g(f(0,0))=(0,1)\neq (0,0)$ and $h(0,0)=(1,0)\neq (0,0)$.
\end{proof}

\section{Toolkit: Ehrenfeucht Fra\"{\i}ss\'e games in finite algebras}\label{sec:EF}
In this section we consider  Ehrenfeucht-Fra\"{\i}ss\'e games on finite algebras.  
While the reader can find a deeper development of Ehrenfeucht-Fra\"{\i}ss\'e games in almost any model theory or finite model theory text, the restriction to finite algebras offers some modest improvements that do not seem to be exposited in such texts.   
Libkin~\cite{lib} for example develops these games in the finite setting, but restricts to relational signatures (possibly with constants).  
Hodges~\cite{hod} presents Ehrenfeucht-Fra\"{\i}ss\'e games for arbitrary signatures (on not-necessarily finite models), but two versions of the game $EF_k({\bf A},{\bf B})$ and $EF_k[{\bf A},{\bf B}]$ (coinciding in relational signatures) are introduced to capture different features, and only un-nested first order formul{\ae} are captured.  
The present section serves as both a preliminary introduction for later use, particularly in results Section~\ref{sec:undecid}, and to make the observation that \emph{finite} algebraic signatures offer a uniform treatment of both algebraic and relational signatures (essentially by way of $EF_k({\bf A},{\bf B})$).  For example, we prove completeness of the $EF_k({\bf A},{\bf B})$ game when it comes to establishing the failure of first order definability amongst finite algebras within uniformly locally finite classes.  

The board of the game consists of two structures ${\bf A}$ and ${\bf B}$, which in this paper will be finite structures in the same arbitrary finite signature (typically algebraic for us, though not necessarily).  The goal of Spoiler is to reveal the two structures as distinct; the goal of Duplicator is to hide any distinction between the two  for as long as possible!

The players play a certain number of rounds.  Each round consists of the following rules.
\begin{enumerate}
\item Spoiler picks a structure; either ${\bf A}$ or ${\bf B}$.
\item Spoiler makes a move by selecting an element of that structure; either $a\in{\A}$ or $b\in{\B}$.
\item Duplicator responds by selecting an element in the other structure.\\
\end{enumerate}

After $n$ rounds of an Ehrenfeucht-Fra\"{\i}ss\'e game, we have moves ($a_1,\dots,a_n$) and ($b_1,\dots,b_n$). 
Duplicator wins the $n$-round game  provided the map that sends each $a_i$ to $b_i$ extends to an isomorphism between the substructures of ${\bf A}$ and ${\bf B}$ generated by $\{a_1,\dots,a_n\}$ and $\{b_1,\dots,b_n\}$ respectively.  If Duplicator can win the $n$-round game regardless of how Spoiler plays, we say that Duplicator has \emph{an $n$-round winning strategy} and we write ${\bf A}\equiv_n{\bf B}$.  In the conventional setting, where positive arity operations are disallowed, there is almost no generation power: it simply means the usual automatic inclusion of any constants in the signature.  

The \emph{quantifier rank} $\qr$ of a formula is defined inductively as follows:
\begin{itemize}
\item If $\phi$ is an atomic formula then $\qr(\phi)=0$;
\item $\qr(\phi_1\vee\phi_2)=\qr(\phi_1\And\phi_2)=\max\{\qr(\phi_1),\qr(\phi_2)\}$;
\item $\qr(\neg\phi)=\qr(\phi)$;
\item $\qr(\forall x\,  \phi)=\qr(\exists x\, \phi)=\qr(\phi)+1$.\\
\end{itemize}
We use the notation $\FO[n]$ for all first order formul{\ae} of quantifier rank up to $n$, and in general use the abbreviation FO for \emph{first order logic}
(recalling that \texttt{FO} refers to the class of first order definable decision problems).

The first fundamental theorem for Ehrenfeucht-Fra\"{\i}ss\'e games in finite model theory is the following; there is no requirement for the signature to be relational here.
\begin{thm}\label{thm:algEF}
The following are equivalent for two finite structures ${\bf A}$ and ${\bf B}$ of the same signature\up:
\begin{enumerate}
\item ${\bf A}\equiv_n{\bf B}$\up; that is, Duplicator has a winning strategy in an $n$-round play of our new game.
\item ${\bf A}$ and ${\bf B}$ agree on $\FO[n]$.
\end{enumerate}
\end{thm}

Theorem \ref{thm:algEF} gives rise to a general methodology for proving inexpressibility results, which we give in the next theorem.  The  theorem is usually restricted to relational signatures (no operations of positive arity) where uniform local finiteness holds always, and so the entire statement can be given as an if and only if.  The proof under uniform local finiteness assumptions is the main new contribution in the section, though the arguments are mostly almost identical, but incorporating Birkhoff's Finite Basis Theorem \ref{BFBT}.
\begin{thm}\label{thm:meth}
A property $\mathcal{P}$ is not expressible in FO if for every $n\in \mathbb{N}$, there exist two finite structures ${\bf A}_n$ and ${\bf B}_n$, such that\up:
\begin{itemize}
\item ${\bf A}_n\equiv_n{\bf B}_n$.
\item ${\bf A}_n$ has property $\mathcal{P}$, and ${\bf B}_n$ does not.
\end{itemize}
When relativised to any uniformly locally finite class of structures, the converse also holds\up: inexpressibility in FO ensures the existence of ${\bf A}_n$ and ${\bf B}_n$.
\end{thm}
    
The following example demonstrates the necessity of some caveat for the ``only if'' part of Theorem \ref{thm:meth}, and therefore a fundamental difference between relational signatures and algebraic signatures.
\begin{example}
Consider the signature of \emph{unars}---unary algebras with a single unary operation $f$---and consider the property $\mathcal{P}$ that $f^n(x)\approx f^{n+1}(x)$ holds for sufficiently large $n$.  Then~$\mathcal{P}$ is not first order definable at the finite level, yet unars with property~$\mathcal{P}$ can be distinguished from unars failing  $\mathcal{P}$ in a one step algebraic Ehrenfeucht-Fra\"{\i}ss\'e game.
\end{example}
\begin{proof}
Assume ${\bf A}$ satisfies $\mathcal{P}$ and  ${\bf B}$ does not.  So there is $b\in B$ such that $f^n(b)$ never coincides with $f^{n+1}(b)$, for any $n$.  Spoiler chooses this $b$, and for any $a\in A$ that duplicator selects we find that $\sg{\bf A}{a}\not\cong \sg{\bf B}{b}$ because for some $n$ the property $f^n(a)=f^{n+1}(a)$ holds.  So ${\bf A}\not\equiv_1{\bf B}$.

Now to show that $\mathcal{P}$ is not first order definable at the finite level.  For this it suffices to work in the relational setting, where $f$ is replaced by its graph: this is because if there is a first order formula defining $\mathcal{P}$ in the algebraic setting, then it may be translated into one defining $\mathcal{P}$ in the relational setting.  Let $\mathbb{A}_k$ denote the digraph consisting of a path of length $k$ with a loop as terminal vertex, and let $\mathbb{B}_k$ denote the digraph consisting of the disjoint union of $\mathbb{A}_k$ with a $k$-cycle.  It is routine exercise in standard  Ehrenfeucht-Fra\"{\i}ss\'e games to verify that for sufficiently large~$k$, Duplicator may survive $n$ rounds of this game; see Figure~3.4 of Libkin~\cite{lib} and accompanying text for example.
\end{proof}

Define the \emph{height} of a term as follows: variables are of height $0$.  If $f$ is a fundamental operation of arity $n$ and the maximum height amongst some set of terms $t_1,\dots,t_n$ is $k$, then $f(t_1,\dots,t_n)$ is of height $k+1$.  
We interpret this as defining the height of nullaries to be $1$.

The following lemma is an algebraic variant of Lemma 3.13 of Libkin \cite{lib}.
\begin{lem}\label{lem:tpfinite}
Let $k,m,n$ be natural numbers.
\begin{enumerate}
\item Every atomic subformula of a rank $k$ formula in $m$ free variables $x_0,\dots,x_{m-1}$ involves at most $m+k$ variables.
\item Up to logical equivalence there are only finitely many distinct rank $k$ formul{\ae} in the $m$ free variables $x_0,\dots,x_{m-1}$ whose terms are of height at most $n$.
\end{enumerate}
\end{lem}
\begin{proof}
Both statements are proved by induction over $k$ (the number $m$ varies: for each $k$ we ask the statement be true of any $m$).  
We omit full details of the proof as it is essentially identical to that in \cite[Lemma~3.13]{lib}, with the one difference occurring in the second item, in consideration of the base case of a rank $0$ formula.  
However the difference is minor: the bound on height of terms continues to ensure that there are only finitely many atomic formul{\ae} in $m$ free variables, and every rank $0$ formula is a Boolean combination of these, and hence the base statement is true.   \end{proof}

\begin{defn}
Let ${\bf A}$ be a finite algebra and $a\in A$. Define the \emph{rank $k$ type} of the element $a$, denoted $\tp_k({\bf A},a)$, to be the set of all  FO-formul{\ae} $\varphi(x)$ of rank $k$ for which ${\bf A}\models \varphi(a)$.
\end{defn}
The set $\tp_k({\bf A},a)$ is always infinite because there is no bound on height of any terms that appear within.  The following lemma contains the primary use of uniform local finiteness.
\begin{lem}\label{lem:finsubset}
Let ${\bf A}$ be a finite algebra and $a\in A$.  Then the rank $k$ type $\tp_k({\bf A},a)$ of $a$ is logically equivalent to a finite subset of $\tp_k({\bf A},a)$.
\end{lem}
\begin{proof}
We use the notation and statement of Birkhoff's Finite Basis Theorem~\ref{BFBT}, specifically the sets $T_{\rm BFB}(\mathsf{HSP}_{\rm fin}({\bf A}),k+1)$ and $\Sigma_{\rm BFB}(\mathsf{HSP}_{\rm fin}({\bf A}),k+1)$, noting that we are using $k+1$ rather than $k$ as in the statement.   Each equation 
\[
f(u_{i_1}(x_0,x_1,\dots,x_k),\dots, u_{i_n}(x_0,x_1,\dots,x_k))\approx u_{i'}(x_0,x_1,\dots,x_k)
\]
 in $\Sigma_{\rm BFB}$ is satisfied universally by ${\bf A}$, and hence the formula 
 \[
 \forall x_1\dots \forall x_k\ f(u_{i_1}(x,x_1,\dots,x_k),\dots, u_{i_n}(x,x_1,\dots,x_k))\approx u_{i'}(x,x_1,\dots,x_k)
 \]
  is a rank $k$ formula satisfied at $x=a$, and so is contained in the set $\tp_k({\bf A},a)$.   
  Let $\Sigma_{\rm BFB}'$ be the set of these formul{\ae}: the equations in $\Sigma_{\rm BFB}$ but with the first variable listed as $x$ and not quantified. 
  Let $\Psi(x)$ be an arbitrary formula in $\tp_k({\bf A},a)$, and let $s(x,x_1,\dots,x_k)\approx t(x,x_1,\dots,x_k)$ be an atomic subformula of $\Psi(x)$, noting that the first statement in Lemma \ref{lem:tpfinite} implies that there are indeed (at most) $k+1$ distinct variables.   
  Then by Birkhoff's Finite Basis Theorem~\ref{BFBT} part~(3), there are $i,j$ such that  ${\bf A}\models s(x,x_1,\dots,x_k)\approx u_i(x,x_1,\dots,x_{k})$ and ${\bf A}\models t(x,x_1,\dots,x_k)\approx u_j(x,x_1,\dots,x_{k})$, and by part~(4), there is a proof of these equalities that proceeds by direct replacement using laws in $\Sigma_{\rm BFB}$.   
  Thus we can replace the atomic subformula $s(x,x_1,\dots,x_k)\approx t(x,x_1,\dots,x_k)$ in $\Psi$ by $u_i(x,x_1,\dots,x_{k})\approx u_j(x,x_{1},\dots,x_{k})$, where direct replacement is used because the value of $x$ is to be fixed as the single element $a$ so cannot be substituted.  
  (When such a replacement is made, we keep $x$ free and implicitly universally quantify any new variables at the atomic occurrence where they are introduced. 
  After renaming bound variables by quantifier depth, this ensures that the quantifier rank remains at most $k$.)
  It follows that $\Psi(x)$ is equivalent modulo $\Sigma'_{\rm BFB}$ to one in which all terms have height at most the maximum height of the terms in $T_{\bf BFB}$.  
  Thus, by the second part of Lemma \ref{lem:tpfinite}, the set $\tp_k({\bf A},a)$ is logically equivalent to the finite subset consisting of $\Sigma_{\rm BFB}'$ along with one representative of each logical-equivalence class of members of $\tp_k({\bf A},a)$ whose terms are of height at most $\max\{\operatorname{height}(t):t\in T_{\rm BFB}\}$.
\end{proof}
The following is essentially Theorem 3.15 of Libkin \cite{lib} in the algebraic setting.  As the (easy) proof is unchanged we omit it, though comment that the formula $\alpha_k(x)$ is of course just the conjunction of the members of the finite subset of $\tp_k({\bf A},a)$ identified as equivalent to $\tp_k({\bf A},a)$ in Lemma \ref{lem:finsubset}.
\begin{lem}\label{lem:type}
For any finite algebra ${\bf A}$ and element $a\in A$ there is a rank $k$ formula $\alpha_k(x)$ such that ${\bf A}\models \alpha_k(a)$ and any similar algebra ${\bf B}$ with $b\in B$ satisfying $\alpha_k(b)$ has $\tp_k({\bf B},b)=\tp_k({\bf A},a)$.
\end{lem}

We shall prove the equivalence of (1) and (2) in Theorem \ref{thm:algEF}, as well as the following \emph{back-and-forth equivalence} (a standard step in this argument, defined here for convenience).
\begin{defn}\label{defn:backNforth}
We inductively define the back-and-forth relations $\simeq_k$ on finite $\mathcal{F}$-structures ${\bf A}$ and ${\bf B}$ as follows.
\begin{itemize}
\item ${\bf A}\simeq_0{\bf B}$ iff ${\bf A}\equiv_0{\bf B}$; that is, ${\bf A}$ and ${\bf B}$ satisfy the same atomic sentences.
\item ${\bf A}\simeq_{k+1}{\bf B}$ iff the following two conditions hold:\hfill\\
\indent {\bf forth}: for every $a\in{\bf  A}$, there exists $b\in{\bf B}$ such that $({\bf A}, a)\simeq_k({\bf B},b)$.\\
\indent {\bf back}: for every $b\in{\bf B}$, there exists $a\in{\bf A}$ such that $({\bf A}, a)\simeq_k({\bf B},b)$.
\end{itemize}
\end{defn}

 The following extension of Theorem \ref{thm:algEF} now follows using the standard proof. 
\begin{thm}
Let ${\bf A}$ and ${\bf B}$ be two finite structures of the same signature. Then the following are equivalent:
\begin{enumerate}
\item ${\bf A}$ and ${\bf B}$ agree on FO[$k$]\up;
\item ${\bf A}\equiv_k{\bf B}$\up;
\item ${\bf A}\simeq_k{\bf B}$.
\end{enumerate}
\end{thm}
\begin{proof}
We omit the proof, as the one given by Libkin for \cite[Theorem~3.18]{lib} (for example) holds close to verbatim, with the only adjustments for algebraic setting having occurred in the adjusted proof of Lemmas~\ref{lem:finsubset} and~\ref{lem:type}  (which are used in the proof of $(1)\implies(3)$), and in the meaning of the relations $\equiv_k$, $\simeq_k$.  We note that the primary difference in the meaning of $\equiv_k$, $\simeq_k$ manifests at the trivial base step: a statement
\[
({\bf A},a_1,\dots,a_n)\simeq_0({\bf B},b_1,\dots,b_n)
\]
(which is the base $\simeq_0$ assumption founding the relation~$\simeq_n$) is equivalent to the subalgebra $\sg{\bf A}{a_1,\dots,a_n}$ being isomorphic to $\sg{\bf B}{b_1,\dots,b_n}$ under the map sending each $a_i$ to the corresponding $b_i$.  But this matches our game-based definition of $\equiv_0$ and of agreement on $\FO[0]$.
\end{proof}

Finally, we prove the final statement in Theorem \ref{thm:meth}, showing completeness of the method for detecting failure of first order definability within uniformly locally finite classes of finite algebras. We do this by proving the contrapositive: a property~$\mathcal{P}$ of a uniformly locally finite class of finite algebras $\mathscr{U}$ is expressible in FO iff there exists a number $k$ such that for every two algebras ${\bf A}$ and ${\bf B}$ in $\mathscr{U}$, if ${\bf A}\in\mathcal{P}$ and ${\bf A}\equiv_k{\bf B}$, then ${\bf B}\in\mathcal{P}$.

\begin{proof}[Proof of final statement in Theorem \ref{thm:meth}.]
Let ${\bf A}$ and ${\bf B}$ be in the uniformly locally finite $\mathscr{U}$.
If $\mathcal{P}$ is expressible by a FO sentence $\Phi$, let $k\coloneqq \qr(\Phi)$. 
If ${\bf A}\in\mathcal{P}$, then ${\bf A}\models\Phi$, and hence for ${\bf B}$ with ${\bf A}\equiv_k{\bf B}$ we have ${\bf B}\models\Phi$. Thus ${\bf B}\in\mathcal{P}$.

Conversely, if ${\bf A}\in\mathcal{P}$ and ${\bf A}\equiv_k{\bf B}$ imply ${\bf B}\in\mathcal{P}$, then any two algebras in $\mathscr{U}$ with the same rank-$k$ type agree on $\mathcal{P}$. 
Now, as $\mathscr{U}$ is a uniformly locally finite class of finite algebras, there is a bound on the size of the $n$-generated algebras and Birkhoff's Finite Basis Theorem~\ref{BFBT} and Lemma~\ref{lem:tpfinite} provide a finite list of sentences $\varphi_1,\varphi_2,\dots,\varphi_m$ of quantifier rank at most $k$, such that every sentence of quantifier rank at most $k$ is equivalent within $\mathscr U$ to one of the $\varphi_i$.
Then every member ${\bf A}$ of $\mathscr{U}$ is determined, up to $\equiv_k$-equivalence, by truth values of the sentences $\varphi_i$:
\[
\theta_{\bf A}\coloneq \bigand_{{\bf A}\models \varphi_i}\varphi_i\And \bigand_{{\bf A}\not\models \varphi_i}\neg\varphi_i.
\]
Hence $\mathcal{P}$ is the disjunction of the at most $2^m$ possible formulas $\theta_{\bf A}$ witnessed by members ${\bf A}$ satisfying~$\mathcal{P}$.
\end{proof}

To conclude this section, we mention a further useful trick for demonstrating non first order definability, this time calling on the classical model theoretic ultraproduct.  This simple but effective idea has been used recently by Cz\'edli \cite[\S2]{cze} for example; we give the proof for completeness.
\begin{lem}\label{lem:czedli}
Let $\mathscr{K}$ be a class of finite structures such that for all large enough~$n$ there are finite structures $\mathbf{A}_n\in \mathscr{K}$ and ${\bf B}_n\notin\mathscr{K}$ for which there is non-principal ultrafilter $\mathscr{U}$ on $\mathbb{N}$ such that the ultraproduct $\prod_{n}{\bf A}_n/\mathscr{U}$ and $\prod_{n}{\bf B}_n/\mathscr{U}$ are isomorphic.  Then $\mathscr{K}$ is not first order definable at the finite level.
\end{lem}
\begin{proof}
Any first order sentence $\Phi$ true of $\mathscr{K}$ is true of all sufficiently large $\mathbf{A}_n$, and then of $\prod_{\mathbb{N}/\mathscr{U}}{\bf A}_n$, and therefore of at least infinitely many of the $\mathbf{B}_n$.  
Hence $\Phi$ does not define~$\mathscr{K}$.
\end{proof}
An example usage of this lemma will be given in Proposition \ref{pro:AC2FO} below.

\section{A first order definable pseudovariety that is not definable by a $\forall^*\exists^*$ sentence.}\label{sec:FO}
The construction of a nonfinitely based finite algebra with variety membership problem in $\FO$ will be based on an idea used in Clark, Davey, Jackson and Pitkethly \cite[Theorem 4.2]{CDJP}: the congruence classes of any congruence on a finite lattice are intervals. 
The result in \cite{CDJP} extends beyond finite algebras to infinite Boolean topological algebras, but in the restriction to finite structures, it suffices to have just a semilattice $\wedge$, as we now prove.  
\begin{thm}\label{thm:FOq}
If $\{{\bf L}_1,\dots,{\bf L}_k\}$ is a finite set of semilattice-based finite algebras, then the class $\mathsf{SP}_{\rm fin}^+(\{{\bf L}_1,\dots,{\bf L}_k\})$ is the class of finite models of a $\forall^*\exists^*\forall^*$ sentence in first order logic.  The same is true for $\mathsf{SP}_{\rm fin}(\{{\bf L}_1,\dots,{\bf L}_k\})$.
\end{thm}
Before proving Theorem \ref{thm:FOq} we prove a lemma.  With minor modifications, this lemma (and then Theorem \ref{thm:FOq}) can  be adapted to include relations.

\begin{lem}\label{lem:onto}
Let $\mathscr{L}$ be the class of finite algebras in some finite signature including a binary operation $\wedge$ with respect to which, all members of $\mathscr{L}$ are semilattices.  Fix any ${\bf L}\in\mathscr{L}$.  Then there is
\begin{itemize}
\item a first order $\exists^*\forall^*$-sentence $\Xi_{\bf L}$ asserting the existence of a surjective homomorphism onto ${\bf L}$,
\item a first order $\exists^*\forall^*$-formula ${\not\equiv_{\bf L}}(x,y)$ \up(written $x\not\equiv_{\bf L}y$\up) in two free variables such that for every member ${\bf K}\in\mathscr{L}$ and every pair $a\neq b$ in  $K$, we have that there is a homomorphism $\phi$ from ${\bf K}$ onto ${\bf L}$ with $\phi(a)\neq \phi(b)$ if and only if ${\bf K}\models a\not\equiv_{\bf L}b$.
\end{itemize}
\end{lem}
\begin{proof}
This is essentially a version of the first part of the proof of \cite[Theorem~4.2]{CDJP}.  Let the universe of ${\bf L}$ be $\{1,2,\dots,\ell\}$.  When there is a surjective homomorphism $\phi:{\bf K}\to {\bf L}$, then the kernel of $\phi$ can be understood entirely using the minimal elements of each congruence class: two elements $a,b$ are equivalent modulo this kernel if and only if they sit above precisely the same minimal elements: $\bigand_{1\leq i\leq \ell} (a\geq x_i\leftrightarrow b\geq x_i)$, where the $x_i$ are variables that  will later be existentially quantified over some larger expression; on satisfaction of this expression, the $x_i$ variables will eventually evaluate as minimal elements in congruence classes.  
We denote this relation by $\equiv_L$ (without boldface), noting that it is clearly an equivalence relation subject to any evaluation of the variables $x_1,\dots,x_\ell$.  Notationally we write $a\equiv_L b$.  The fact that the kernel is a congruence on ${\bf K}$ can be asserted by the property that for each operation $f$, of arity $n$ say, we have
\[
\forall a_1,\dots,a_n,b_1,\dots,b_n\ \Big(\bigand_{j=1}^n a_j\equiv_L b_j\Big)\rightarrow f(a_1,\dots,a_n)\equiv_L f(b_1,\dots,b_n)
\]
The fact that $\phi$ maps into ${\bf L}$ can also be asserted using $\equiv_L$: for each fundamental operation $f$ (again, of arity $n$, say), and for each value $f(i_1,\dots,i_n)=j$ in ${\bf L}$ we require $f(x_{i_1},\dots,x_{i_n})\equiv_L x_j$.
 To ensure $\phi$ is surjective, we can additionally assert that the $x_i$ are not equivalent, and that all elements are $\equiv_L$ to at least one:
\[
\bigand_{i<j}\neg(x_i\equiv_L x_j)\And
\forall z\,\bigvee_{i=1}^{\ell}z\equiv_L x_i.
\]
We take the conjunction of all of these formul{\ae}, and existentially quantify $x_1,\dots,x_\ell$ to create a $\exists^*\forall^*$ sentence $\Xi_{\bf L}$ asserting that there is a surjective homomorphism from ${\bf K}$ onto ${\bf L}$. 
We wished to define the corresponding relation $\not\equiv_{\bf L}$ (now with boldface and negated).  To do this, we introduce two free variables $x,y$ and include $\neg(x\equiv_L y)$ as a further conjunct within~$\Xi_{\bf L}$.
\end{proof}
Now we may prove Theorem \ref{thm:FOq}.
\begin{proof}[Proof of Theorem \ref{thm:FOq}]
An algebra ${\bf A}$ is in the universal Horn class of $\{{\bf L}_1,\dots,{\bf L}_k\}$ if and only if  there is at least one homomorphism from ${\bf A}$ to one of the ${\bf L}_i$ and the following ``separation properties'' hold: whenever $a,b\in A$ have $a\neq b$, there is an $i\in\{1,\dots,k\}$ and homomorphism from ${\bf A}$ to ${\bf L}_i$ with $\phi(a)\neq \phi(b)$.  As $\{{\bf L}_1,\dots,{\bf L}_k\}$ is a finite set of finite algebras, we may enumerate the set of all subalgebras ${\bf L}'_1,\dots, {\bf L}_m'$.  The following sentence (referring to the sentences defined in Lemma \ref{lem:onto}) asserts the required properties:
\[
\Big(\bigvee_{1\leq i\leq m}\Xi_{{\bf L}'_i}\Big)\And \forall a\forall b \Big(a\not\approx b\rightarrow \bigvee_{1\leq i\leq m}a\not\equiv_{{\bf L}'_i} b\Big)
\]
Because $\Xi$ and $\not\equiv$ are $\exists^*\forall^*$, this disjunction is a $\forall^*\exists^*\forall^*$-sentence.  Once the semilattice axioms for $\wedge$ have been incorporated, we have a single sentence defining $\mathsf{SP}_{\rm fin}^+(\{{\bf L}_1,\dots,{\bf L}_k\})$.  The sentence for $\mathsf{SP}_{\rm fin}(\{{\bf L}_1,\dots,{\bf L}_k\})$ is the same except that we do not  include the $\bigvee_{1\leq i\leq m}\Xi_{{\bf L}'_i}$ conjunct.
\end{proof}

As a corollary of Theorem \ref{thm:FOq} we prove that the first order JP Problem is true for algebras with a semilattice operation; this follows already from Willard's Finite Basis Theorem \cite{wil2} or his earlier semilattice-based algebra result~\cite{wil}, but the relative simplicity of our proof is suggestive of the possibility of more accessible results of this kind in the finite realm.
\begin{cor}\label{cor:willard}
If ${\bf A}$ is a finite semilattice-based algebra generating a pseudovariety with only finitely many subdirectly irreducible members, then the pseudovariety of~${\bf A}$ is definable in first order logic amongst finite structures.
\end{cor}
\begin{proof}
Let ${\bf A}_1$, \dots, ${\bf A}_k$ be a complete list of the subdirectly irreducible members of the pseudovariety of ${\bf A}$ (all finite).  Because every algebra is a subdirect product of its subdirectly irreducible quotients,  the pseudovariety generated by ${\bf A}$ coincides with $\mathsf{SP}_{\rm fin}(\{{\bf A}_1,\dots,{\bf A}_k\})$.  Now apply Theorem \ref{thm:FOq}.
\end{proof}

The \emph{flat extension} of an algebra ${\bf A}$ is the algebra $\flat({\bf A})$ formed over the universe $A\cup\{0\}$ (considered as a disjoint union; rename elements of $A$ if necessary) and with the signature of ${\bf A}$ augmented by a further binary operation $\wedge$ (again, rename existing operations if necessary).  The existing operations of ${\bf A}$ are extended to $0$ by making $0$ an absorbing element, and the operation $\wedge$ is defined by 
\[
x\wedge y =\begin{cases} x&\mbox{ if }x=y\\
0&\mbox{ otherwise.}\end{cases}
\]
The variety $\mathsf{HSP}(\flat({\bf A}))$ generated by the flat extension of a finite algebra ${\bf A}$ (or even partial algebra), has some relationship to the universal Horn class $\mathsf{SP}^+({\bf A})$ of the original algebra ${\bf A}$ \cite{jac:flat,wil0}.  The relationship is particularly precise in the case that~${\bf A}$ has a binary term function acting as a projection: that is, there is a binary term function $\tr$ such that  $x\tr y=y$ for every $x$ and $y$.  In this situation, the subdirectly irreducible members of $\mathsf{HSP}(\flat({\bf A}))$ are precisely the flat extensions of the algebras in the universal Horn class of ${\bf A}$; see \cite[\S5]{jac:flat}.  Much of this can be explained using  the ternary term $(x\wedge y)\tr z$.  On the algebra $\flat({\bf A})$, this term acts as follows:
\[
(x\wedge y)\tr z=\begin{cases}
z &\mbox{ if }x=y\neq 0\\
0&\mbox{ otherwise}.
\end{cases}
\]
In \cite[\S5]{jac:flat} a ternary term with this property is called a \emph{pointed semi-discriminator function} and an algebra is said to be a \emph{pointed semi-discriminator algebra} if it has an absorbing element $0$ and a term function acting as a pointed semi-discriminator function.  It is routine to show that (up to isomorphism), pointed semi-discriminator algebras are exactly those that are term equivalent to flat extensions of partial algebras with a binary projection operation $\tr$.
A variety $\mathscr{V}$ is a \emph{pointed semi-discriminator variety} if there is a ternary term $p(x,y,z)$ such that $\mathscr{V}$  can be generated by algebras on which the term $p$ acts as a pointed semi-discriminator operation; equivalently, if the subdirectly irreducible members of $\mathscr{V}$  are pointed semi-discriminator algebras with respect to $p$.  It is shown in \cite[Theorem 5.8]{jac:flat} that a complete axiomatisation for the pointed semidiscriminator variety over all partial algebras in signature $\mathscr{L}$ is given by the semilattice laws for $\wedge$:
\[
x\wedge (y\wedge z)\approx (x\wedge y)\wedge z,\ x\wedge y\approx y\wedge x,\ x\wedge x\approx x,
\]
 the right normal band laws for  $\tr$:
\[
x\tr (y\tr z)\approx (x\tr y)\tr z,\ x\tr (y\tr z)\approx y\tr (x\tr z),\ x\tr x\approx x,
\]
the laws
\[
(x\wedge y)\tr y\approx x\wedge y,\ (x\tr y)\wedge z\approx x\tr (y\wedge z)
\]
and the laws 
\[
x\tr f(x_1,\dots,x_i,\dots,x_n)\approx f(x_1,\dots,x\tr x_i,\dots,x_n)
\]
for each operation symbol $f\in \mathscr{L}\cup\{\wedge,\tr\}$ (of arity $n$) and each $1\leq i\leq n$.  The lattice of subvarieties of this variety is isomorphic to the lattice of  universal Horn subclasses of the class of all partial algebras in the signature $\mathscr{L}$ (\cite[Corollary~5.4]{jac:flat}).
There is no loss of generality to these facts if the element $0$ is included in the signature as a distinguished nullary satisfying the absorption laws for $0$ with respect to all operations: all finite models and all subdirectly irreducible members of a pointed semidiscriminator variety have this $0$.  We use the suffix ``with $0$'' to describe this amendment; it is included in a nonessential way in the next lemma.

\begin{lem}\label{lem:si}
The pointed semidiscriminator variety with $0$ over any finite signature $\mathscr{L}$ has definable cmi congruences.  More precisely, if $\rho_z(u,v)$ denotes the formula
\[
[((u\wedge v)\tr z\approx z) \vee (u\tr z\not\approx z\And v\tr z\not\approx z)],
\]
then the formula 
\[
\pi_{x,y}(u,v)\coloneqq x\approx y\vee (x\not\approx x\wedge y\And\rho_x(u,v))\vee(x\approx x\wedge y\And\rho_y(u,v))
\]
defines cmi congruences in the variety of pointed semi-discriminator algebras over~$\mathscr{L}$.
\end{lem}
\begin{proof}
This is established in the proof of Theorem 3.1(3) in \cite{jac:flat}: see the argument starting in paragraph three of the proof, where this congruence is described in written form, and it is subsequently verified that the corresponding quotient is subdirectly irreducible.  Note that the proof in \cite{jac:flat} is in a more general context and uses expressions such as $\lambda(c\wedge d)=a$, where $\lambda$ is an arbitrary translation.  In the pointed semidiscriminator setting, the property that there exists a translation $\lambda$ such that $\lambda(c\wedge d)=a$ is equivalent to $(c\wedge d)\tr a=a$, which in the logical syntax given in the lemma statement appears in subexpressions such as $((u\wedge v)\tr x\approx x)$.
\end{proof}
A pointed semidiscriminator variety is not division ordered, but there is a strong similarity between the cmi congruence defined in Lemma~\ref{lem:si} and that defined in the proof of Proposition~\ref{pro:divisionordered}.  The condition $x\not\approx x\wedge y$ is playing the role of $y\not\divides x$, while $\rho_x(u,v)$ is easily seen to define the syntactic congruence of the set defined by $\{z\mid z\tr x\not\approx x\}$, with a clear similarity to the set $J_x=\{z\mid z\not\divides x\}$ used in the proof of Proposition~\ref{pro:divisionordered}.

We can now give the main result of the section.  While we omit details, the result can be established more generally in the case that ${\bf L}$ is a partial algebra, provided the semilattice $\wedge$ and the projection term $\tr$ are total.
\begin{thm}\label{thm:FO}
Let ${\bf L}$ be a finite semilattice based algebra with a projection term.  Then the pseudovariety generated by $\flat({\bf L})$ with distinguished $0$ is definable by a $\forall^*\exists^*\forall^*$ sentence of first order logic.
\end{thm}
\begin{proof}
The subdirectly irreducible members of the pseudovariety of $\flat({\bf L})$ are just the flat extensions (with $0$ distinguished) of the members of the universal Horn class of ${\bf L}$.  By Theorem~\ref{thm:FOq}, the class $\mathsf{SP}^+_{\rm fin}({\bf L})$ is first order definable by a $\forall^*\exists^*\forall^*$ sentence $\forall x\forall y\ \Phi(x,y)$ where $\Phi(x,y)$ is an $\exists^*\forall^*$ formula.  
The class of flat extensions of members of this class can then be defined in first order logic by asserting that the structure is flat, with bottom element $0$ (this uses only universal quantifiers), and that $\forall x\forall y\ \Phi(x,y)$ holds for all elements that are not $0$: every variable $z$ in $\forall x\forall y\ \Phi(x,y)$ must be restricted to nonzero values: replace $\forall z\,\psi$ by 
$\forall z\,(z\neq0\rightarrow\psi)$,
and $\exists z\,\psi$
by $\exists z\,(z\neq0\And\psi)$.  
We must also additionally assert that all fundamental operations in the signature of ${\bf L}$ map nonzero inputs to nonzero outputs: $\forall\vec z\,
\left(\bigand_i z_i\not\approx 0\ \rightarrow\ f(\vec z)\not\approx  0\right).$
(These amendments are where having $0$ is a constant is required in our proof; otherwise it would be necessary to existentially quantify the existence of a zero element, adding an extra layer of quantifier alternation to the sentence.)  
Thus we have a sentence $\Psi$ defining, amongst nontrivial algebras, the subdirectly irreducible members of
$\mathsf{HSP}_{\rm fin}(\flat({\bf L}))$.
We could now try to apply the one-to-many first order reduction of Theorem~\ref{thm:fodecomp}, but this calls on a linear order, which is acceptable from a complexity-theoretic perspective, but steps outside of pure first order logic of the algebras in consideration.  We can avoid this entirely using the formula $\pi_{x,y}(u,v)$ of Lemma~\ref{lem:si}.  
Let~$\pi_{a,b}$ denote the congruence relation defined by $\pi_{a,b}(u,v)$.  For a model ${\bf K}$ we wish to assert that whenever $a,b\in K$ have $a\neq b$ then ${\bf K}/\pi_{a,b}\models \Psi$.  
Let $\Psi^\flat(a,b)$ denote the result of replacing every instance of equality $u\approx v$ in $\Psi$ by $\pi_{a,b}(u,v)$ (here $u$ and $v$ refer to any terms related by equality in $\Psi$).  Then ${\bf K}/\pi_{a,b}\models \Psi$ for every $a\neq b$ is equivalent to 
\[
{\bf K}\models \forall a\forall b\ \neg(a\approx b)\rightarrow \Psi^\flat(a,b).
\]
Thus the sentence $\forall a\forall b\ \neg(a\approx b)\rightarrow \Psi^\flat(a,b)$ along with the equational axioms defining pointed semidiscriminator varieties with $0$, define $\mathsf{HSP}_{\rm fin}(\flat({\bf L}))$ amongst finite algebras.
\end{proof}

There are examples of lattices ${\bf L}$ for which $\mathsf{SP}({\bf L})$ has no finite basis for its universal Horn theory: Belkin~\cite{bel} has shown that the 10-element lattice $M_{3{-}3}$ has this property for example (noting that $\mathsf{SP}({\bf L})$ and $\mathsf{SP}^+({\bf L})$ coincide for any lattice).  
As lattices are semilattice based and have a projection term $t(x,y)\coloneqq y+(x\cdot y)$ (using $+$ and $\cdot$ for join and meet respectively, to avoid conflict with the logical connective $\vee$ and our use of $\wedge$ for the flat semilattice operation), we have that $\mathsf{HSP}(\flat({\bf L}))$ (with $0$) has no finite axiomatisation in first order logic.  
Theorem \ref{thm:FO} and Corollary~\ref{cor:compton} give us the flagship results.
\begin{thm}\label{thm:main}
If ${\bf L}$ is a finite lattice without a finite basis for its universal Horn theory\up, then $\flat({\bf L})$ with $0$ generates a nonfinitely based variety for which the pseudovariety of finite members is first order definable by a $\forall^*\exists^*\forall^*$ sentence amongst finite algebras, but not by any $\forall^*\exists^*$ sentence, nor by any finite set of pseudoequations.
\end{thm}
\begin{proof}
Definability by a $\forall^*\exists^*\forall^*$ sentence of the pseudovariety (amongst finite algebras) comes directly from Theorem~\ref{thm:FO}, while the nonfinite basis property is already discussed.  

Now we consider the pseudoequation claim.  
Let $\Sigma$ be a finite set of pseudoequations holding in $\flat({\bf L})$ with $0$; we show that $\Sigma$ does not capture the pseudovariety of $\flat({\bf L})$ with $0$.  Let~$k$ be the number of variables appearing in $\Sigma$.
Because ${\bf L}$ has no finite basis for its universal Horn theory, for each $n$ there exists a lattice ${\bf L}_n$ that lies outside of the universal Horn class of ${\bf L}$ but whose $n$-generated sublattices lie within the universal Horn class of ${\bf L}$.  
Because ${\bf L}$ has a finite basis for its equational theory (for example, McKenzie~\cite{mck:lat}), there is no loss of generality in assuming that the lattices ${\bf L}_n$ lie within the variety of ${\bf L}$; take $n$ at least as large as both $k$ and the number of variables in the finite basis for the variety of ${\bf L}$. 
By the local finiteness of this variety, there is moreover no loss of generality in assuming that each ${\bf L}_n$  is finite: otherwise we may replace ${\bf L}_n$ with the subalgebra generated by any tuple that witnesses failure of a universal Horn sentence satisfied by ${\bf L}$.
Then $\flat({\bf L}_n)$ lies outside of the pseudovariety of $\flat({\bf L})$, but its $n$-generated subalgebras lie within the pseudovariety of $\flat({\bf L})$.  It follows that $\flat({\bf L}_k)$ satisfies $\Sigma$, showing that $\Sigma$ is not a basis by pseudoequations, as required.

As equations are pseudoequations, and no finite basis by pseudoequations exists, Corollary~\ref{cor:compton} then shows that no $\forall^*\exists^*$ sentence defines the pseudovariety of $\flat({\bf L})$ with $0$.
\end{proof}
Theorem~\ref{thm:main} gives us our simultaneous failure of the $\mathsf{S}$, $\mathsf{SP}$ and $\mathsf{HSP}$ preservation theorems at the finite level: each would entail definability by a kind of universal sentence, but the theorem shows that no $\forall^*\exists^*$ sentence is available.  
The theorem extends to any finite signature finite algebra ${\bf A}$ without a finite basis for its universal Horn theory (and with a projection and semilattice term), provided that the analogues of the algebras ${\bf L}_n$ in the proof (for each $n$) that are guaranteed by the nonfinite basis property can be selected as finite.
There are no known examples of finite algebras ${\bf A}$ whose universal Horn class has no finite basis and where the corresponding algebraic witnesses cannot be chosen to be finite; such an algebra would provide counterexamples to a universal Horn class variant of the ES problem, which in turn would yield a counterexample to the true ES problem; see~\cite[\S7.4]{jac:flat}.

As alluded to in Subsection \ref{subsec:preservation}, it would be interesting to see whether the pseudovariety of an example such as $\mathsf{HSP}_{\rm fin}(\flat(M_{3{-}3}))$ can be defined by a finite set of equation systems in the sense of Higgins and the second author \cite{higjac,higjac23}.

\section{Pseudovariety membership complexity: Constraint Satisfaction Problems}\label{sec:comp}
There has been little exploration of fine level membership for $\varmem(*,\mathbf{A})$ within tractable cases.  In \cite{jacmcn}, the authors show that $\varmem(*,\mathbf{L})$ is in \texttt{NL} for Lyndon's algebra ${\bf L}$.  In \cite{VGK} it is shown that a certain six-element semigroup $AC_2$ (also known as $A_2^g$ \cite{leezha}, and not to be confused with the complexity class $\texttt{AC}^2$) has $\varmem(*,AC_2)$ solvable in polynomial time.  At the end of this section we show that their characterisation of the pseudovariety of $AC_2$ in fact leads to $\texttt{NL}$ solvability of $\varmem(*,AC_2)$, but that it does not lie in \texttt{FO}. \ In the present section we translate CSP complexity into $\varmem$ complexity, thereby providing a wealth of examples within tractable complexity.  For example, in \cite[\S10]{JKN} it is shown that there are infinitely many  tractable complexity classes represented by constraint problems of the kind we translate: as well as commonly encountered classes such as \texttt{L}, \texttt{NL}, there are the classes $\Mod_p(\texttt{L})$ for any prime $p$, and problems of complexity in the $\cap/\cup$ closure of any combination of these.

Let $\mathbb{A}$ be a finite relational structure of finite signature.  
The constraint satisfaction problem over $\mathbb{A}$ is the computational problem defined as follows.

\fbox{\parbox{0.9\textwidth}{\noindent $\CSP(\mathbb{A})$\\
Instance: a finite structure $\mathbb{B}$ of the same signature as $\mathbb{A}$.\\
Question: is there a homomorphism from $\mathbb{B}$ into $\mathbb{A}$?}}
\medskip

We direct the reader to the collection \cite{dag} for a small sample of the wealth of literature on the $\CSP(\mathbb{A})$ problem.  
A \emph{nontrivial} $\CSP$ will mean one that has at least one NO instance.

The main result in this section is the following theorem.
\begin{thm}
For every finite relational structure $\mathbb{A}$ for which $\CSP(\mathbb{A})$ is nontrivial, there is a finite algebra $\mathbf{A}'$ \up(not typically on the same universe\up) such that  $\CSP(\mathbb{A})$ and $\varmem(*,{\bf A}')$ are first order equivalent.
\end{thm}
\begin{proof}
The proof is mostly a corollary of constructions in \cite{jac:flat} and \cite{jactro13}, though the first order equivalence requires fresh arguments.  

It is shown in \cite[\S4]{jactro13} that for every finite relational structure $\mathbb{A}$ of finite signature, there is a structure $\mathbb{A}^\sharp$ of the same signature with the property that for all $\mathbb{B}$ we have $\mathbb{B}\rightarrow\mathbb{A}$ if and only if $\mathbb{B}\in\mathsf{SP}^+(\mathbb{A}^\sharp)$.  
Thus every fixed finite template CSP is a universal Horn class membership problem: they are identical as problems.  

Any relational structure $\mathbb{A}$ gives rise to a partial algebra $\pi(\mathbb{A})$ by replacing each fundamental relation by the projection partial function: if $R$ has arity $n$, then it is replaced by the first projection function on domain  $R^\mathbb{A}\subseteq A^n$.  
To ensure that we arrive at a pointed semidiscriminator variety we also add a projection term~$\tr$ and use the notation $\ps(\pi(\mathbb{A}))$ to denote the flat extension of the result of adjoining the second projection $\tr$ to $\pi(\mathbb{A})$. 
We mention that $\tr$ is just the projection from a degenerate total binary relation, which could alternatively be added as a fundamental relation on all relational structures, with no impact on $\CSP$ complexity.
As in the previous section, we use $0$ for the adjoined absorbing zero element in the construction of $\ps(\pi(\mathbb{A}))$ from $\pi(\mathbb{A})$; this can either be explicitly included as a constant in the signature, or otherwise specified in first order logic as the semilattice minimum element.

In \cite[\S7.1]{jac:flat} it is explained that the subdirectly irreducible members of $\mathsf{HSP}(\ps(\pi(\mathbb{A})))$ are just the isomorphism closure of $\ps(\pi(\mathbb{B}))$ for $\mathbb{B}\in\mathsf{SP}^+(\mathbb{A})$.  The translations between~$\mathbb{B}$ and $\ps(\pi(\mathbb{B}))$ and back again are easily seen to be first order reductions, 
so the combined effect of these observations is that every problem of the form $\CSP(\mathbb{A})$ is first order equivalent to 
$\varsimem(*,\mathbf{A}')$, where  $\mathbf{A}'$ is the algebra $\ps(\pi(\mathbb{A}^\sharp))$.  As $\varsimem(*,\mathbf{A}')$ is a restriction of $\varmem(*,\mathbf{A}')$, to complete the proof it suffices to find a first order reduction from $\varmem(*,\mathbf{A}')$ to $\CSP(\mathbb{A})$. 

It would suffice to reduce $\varmem(*,\mathbf{A}')$ to $\varsimem(*,\mathbf{A}')$, but it is easier to reduce all the way to $\CSP(\mathbb{A})$.  
The strategy is as follows.  We are given a general instance $\mathbf{B}$ of $\varmem(*,\mathbf{A}')$, and there is no loss of generality in assuming that~$\mathbf{B}$ satisfies the finitely many equations that define the pointed semidiscriminator variety of appropriate signature.  
Now $\pi(\mathbb{A}^\sharp)$ satisfies implications $f(x_1,\dots,x_n)\approx f(x_1,\dots,x_n)\rightarrow f(x_1,\dots,x_n)\approx x_1$ asserting that each of its operations are partial projections.  
Each of these finitely many laws also translates to an equation that holds on $\ps(\pi(\mathbb{A}^\sharp))$ (see \cite[Lemma 5.9]{jac:flat}), and as first order formul{\ae}, there is also no loss of generality in assuming that $\mathbf{B}$ satisfies these.  
In this case, all of the subdirectly irreducible quotients of $\mathbf{B}$ (one for each pair $a\neq b$ in $B$; denote them by $\mathbf{B}_{a,b}$) will be isomorphic to algebras of the form $\ps(\pi(\mathbb{B}_{a,b}))$, where $\mathbb{B}_{a,b}$ is an input to $\CSP(\mathbb{A})$.  
We have $\mathbf{B}\in \mathsf{HSP}({\bf A}')$ if and only if $\mathbb{B}_{a,b}$ is a YES instance of $\CSP(\mathbb{A})$ for each $a\neq b$ in~$B$.  
The reader satisfied with a one-to-many first order reduction can now simply use the one-to-many reduction supplied by Theorem \ref{thm:fodecomp}.  
But we can do slightly better, as the property that all of the structures $\mathbb{B}_{a,b}$ are YES instances of  $\CSP(\mathbb{A})$ is equivalent to the single instance $\dot\bigcup_{a\neq b\text{ in }B}\mathbb{B}_{a,b}$ being a YES instance of $\CSP(\mathbb{A})$, and as we now show, this structure can be made as a single first order query.  
We make essential recourse to the implicit linear order $<$ available in the first order query definition. 

Let $\pi_{a,b}(u,v)$ be the formula defining cmi congruences.  
The universe of $\mathbb{B}'$ will be defined as a subset of $B^3$, consisting of all tuples $(a,b,c)$ where $a\neq b$, and where 
 $c$ is a $<$-minimum element in an equivalence class of the congruence $\pi_{a,b}(u,v)$ that is not equivalent to $0$: so, $\forall d\, (\pi_{a,b}(c,d)\rightarrow c\leq d)$ and $\neg\pi_{a,b}(c,0)$.  
Each relation $R$ (of arity $n$ say) corresponded to an operation $f_R$ in the signature of $\mathbf{B}$ of the same arity, and we define $R$ on our defined universe as all tuples 
$((a_1,b_1,c_1),\dots,(a_n,b_n,c_n))$ where $a_1=a_2=\dots=a_n$, $b_1=b_2=\dots=b_n$ and $\pi_{a_1,b_1}(c_1,f_R^\mathbf{B}(c_1,\dots,c_n))$.  
\end{proof}

We finish this  section by proving the claims around the semigroup $AC_2$ at the start of section; the  proof of the first claim simply involves verifying that the algorithmic check from~\cite{VGK} can be performed on a nondeterministic Turing machine in logspace.  The  proof of the second claim is by an application of Lemma~\ref{lem:czedli}.
\begin{pro}\label{pro:AC2}
$\varmem(*,AC_2)$ is in \texttt{NL}.
\end{pro}
\begin{proof}
A polynomial time algorithm for $\varmem(*,AC_2)$ is given in Section 3 of~\cite{VGK}, and consists of checking finitely many equations (which is first order, hence in \texttt{L}), then verifying that the subsemigroup generated by idempotents has trivial subgroups.  
We may verify failure of this last property in nondeterministic logspace as follows. 
First note that the single equation $x^2\approx x^4$ ensures that all subgroups have exponent $2$ or are trivial, and we tacitly assume this throughout.  
Next observe that the property that there is a nontrivial subgroup contained in the subsemigroup generated by the idempotents is equivalent to the existence of idempotents $e_1,e_2,\dots,e_k$ with the product $a\coloneqq e_1\dots e_k$ having $a^3=a$ but $a\neq a^2$.  
We can discover this in nondeterministic logspace as follows.  
We use three small sections of work tape, each logarithmic in size.  
We have a section of tape to record a current idempotent guess~$e$.  
We have a small section of work tape to record the current value of a product~$a$; initially the value of $a$ is the first idempotent guess $e$.  
We also have a small section to check if the current product value~$a$ satisfies $a^3=a$ but $a\neq a^2$.  
The algorithm proceeds by nondeterministically guessing an element $e$, checking if $e$ is idempotent (this only involves storing the current guess on the work tape and consulting the input tape to see if the square of the listed element is identical to what we stored as our guess).  
If the current guess~$e$ is idempotent, we replace the current product value~$a$ by the new value $a\cdot e$.  
Note that we record just the value of this product, not the previous values of~$a$ and~$e$.  
Then we check the input to see if $a^3=a$ but $a\neq a^2$: if $a=a^2$ then this test has failed, but otherwise we may write the value of $a^2$ to the third section of our working tape, and then check if $a\cdot a^2=a$.  
If $a\cdot a^2\neq a$ then we continue our search.  
If $a\cdot a^2= a$ then the input algebra is not in the variety of~$AC_2$.  
Because co-$\texttt{NL}=\texttt{NL}$ we have  $\varmem(*,AC_2)$ is in \texttt{NL}.
\end{proof}
To complement Proposition \ref{pro:AC2} we now show that this example does not provide a counterexample to the first order ES problem.
\begin{pro}\label{pro:AC2FO}
$\varmem(*,AC_2)$ is not in $\texttt{FO}$.
\end{pro}
\begin{proof}
We give sketch details only, and direct the reader to \cite{VGK} or to any semigroup theory text (such as Clifford and Preston~\cite{clipre} or Howie~\cite{how}) for definitions of Rees matrix semigroups and their isomorphism theory.  
The approach will be via Lemma~\ref{lem:czedli}.

Let $C_2=\{e,g\}$ be a copy of the $2$-element group with $g^2=e$.  For each $n>2$, consider the Rees matrix semigroup ${\bf A}_n$ over~$C_2$ whose sandwich matrix is as follows:
\[
\left[\begin{matrix}
e&e&0&0& \cdots &0&0\\
0&e&e&0&\cdots &0&0\\
0&0&e&e&\cdots &0&0\\
\vdots&\vdots&\vdots&&\ddots&\vdots&\vdots\\
0&0&0&0&\cdots &e&e\\
g&0&0&0&\cdots &0&e\\
\end{matrix}\right].
\] 
It is useful to think of this sandwich matrix as recording a labelled bipartite graph between the rows and the columns; $0$ entries correspond to no edge, while nonzero edges are labelled by their value (either $e$ or $g$).  
Evidently the bipartite graph for~${\bf A}_n$ is a $2n$-cycle; all labels are $e$ except for one that is $g$.  
The underlying unlabelled graph is definable algebraically by taking the vertices as the $\mathscr{L}$-classes and $\mathscr{R}$-classes of nonzero elements, and the edges correspond to nonzero products between vertices of these two kinds of equivalence class.
More precisely, there is an edge between the $\mathscr{L}$-class of $a\neq 0$ and the $\mathscr{R}$-class of $b\neq 0$ if and only if $ab\neq 0$.
The property of being completely $0$-simple over $C_2$ is a first order property, so any ultraproduct of the family $\{{\bf A}_n\mid n>2\}$ will also be completely $0$-simple over $C_2$, and the underlying bipartite graph will be an ultraproduct of underlying graphs.  
As all vertices in the underlying graphs have degree $2$, this will be true of the ultraproduct. 
For any fixed $m$, the property of having a cycle of length $m$ is also a first order property, and as the bipartite graph for all sufficiently large ${\bf A}_n$ has no cycle of length at most $m$ (we need only have $n>\lceil m/2\rceil$), it follows that an ultraproduct over a nonprincipal ultrafilter will have no cycles at all: the graph will be a disjoint union of two-way infinite paths.

Now let ${\bf B}_n$ denote the Rees matrix semigroup over $C_2$ with the same sandwich matrix as for ${\bf A}_n$ except that the one entry labelled $g$ is replaced by $e$.  
It is routine to show that ${\bf B}_n$ is in $\mathsf{HSP}(AC_2)$ and that~${\bf A}_n$ is not (it fails the condition from~\cite{VGK} used in the proof of Proposition~\ref{pro:AC2} for example).  
However the underlying unlabelled bipartite graph for ${\bf B}_n$ is identical to that of ${\bf A}_n$, only the edge labelling differs.
It follows that if $\mathscr{U}$ is any nonprincipal ultrafilter over $\{3,4,5,\dots\}$, then the ultraproducts $\prod_{n}{\bf A}_n/\mathscr{U}$ and $\prod_{n}{\bf B}_n/\mathscr{U}$ have identical underlying unlabelled bipartite graphs: the aforementioned disjoint unions of two-way infinite paths.  
Only the edge labels differ, however in the absence of cycles, the general isomorphism theory of Rees matrix semigroups (for example, \cite[Theorem 3.4.1, pp. 81--82]{how})  enables us to make all edge labels uniformly equal to $e$, leaving trivially isomorphic semigroups.  
By Lemma \ref{lem:czedli} we have that $\mathsf{HSP}_{\rm fin}(AC_2)$ is not definable by a first order sentence.
\end{proof}
A more detailed exploration of the possible first order definability of well known nonfinitely based finite semigroups and other algebras would be interesting.  Those where the pseudovariety is already known to be intractable we cannot have first order definability, though for semigroups, this is currently restricted to the examples of \cite{jacmck}, examples relating to ${\bf B}_2^1$  \cite{jac:SAT}, the examples of \cite{KKP} and the combination of \cite{jac:SAT,KKP} given in~\cite{jac:26}, a tiny fraction of the known nonfinitely based examples.

\section{Undecidability of complexity of finite membership.}\label{sec:undecid}
In a groundbreaking series of papers  \cite{mck1,mck2,mck3}, Ralph McKenzie showed that there is no algorithm to decide if a given finite algebra has a finite basis for its equational theory, thereby solving Tarski's long standing Finite Basis Problem.  
Later, Willard \cite{wil} showed how the construction $A(\mathcal{T})$ (here $\mathcal{T}$ is a Turing machine program) of the second of these papers \cite{mck2} could be used to the same effect, thereby avoiding the much more challenging construction $F(\mathcal{T})$ of \cite{mck3}.  
An exposition of this amazing series of constructions can be found in \cite[Sections 7.10 and 7.11]{ALV}, though we have followed the notation of the original. It is natural to ask if the finite model theoretic version of Tarski's Finite Basis Problem is also undecidable, especially in view of our examples from Section \ref{sec:FO}.  
In this section we give a standalone proof demonstrating that McKenzie's $A(\mathcal{T})$ construction does indeed show the undecidability of determining whether the pseudovariety of a finite algebra is first order definable at the finite level.  
\begin{thm}
The following problem is undecidable\up: given a finite algebra ${\bf A}$, determine if the pseudovariety of ${\bf A}$ is first order definable at the finite level.
\end{thm}
\begin{proof}
We show that the pseudovariety of $A(\mathcal{T})$ is first order definable if and only if a Turing machine executing program $\mathcal{T}$ on the blank tape eventually halts.  
For the ``if'' direction we need only know the fact that $A(\mathcal{T})$ has a semilattice operation and that when the Turing machine program $\mathcal{T}$ eventually halts on the blank tape, the variety of~$A(\mathcal{T})$ contains only finitely many subdirectly irreducibles: this is \cite[Theorem~5.1]{mck2}.  
Our Corollary~\ref{cor:willard} above implies first order definability of the pseudovariety of~$A(\mathcal{T})$ at the finite level, though of course we could also have used the much more powerful~\cite{wil} or \cite{wil2}.

Now we must show that when $\mathcal{T}$ does not halt, then the pseudovariety of $A(\mathcal{T})$ is not first order definable.  For this we need more details of $A(\mathcal{T})$ and of some of the subdirectly irreducibles in the pseudovariety of $A(\mathcal{T})$.  The algebra $A(\mathcal{T})$ has a very large number of operations and elements, but our arguments will mainly centre on the multiplication operation, which is remarkably simple.  There is a multiplicative $0$ element, and the only nonzero products exist amongst just 7 of the elements: $1,2,H,C, \bar{C},D,\bar{D}$, with only $1\cdot C=C$, $1\cdot \bar{C}=\bar{C}$, $2\cdot D=H\cdot C=D$ and $2\cdot \bar{D}=H\cdot \bar{C}=\bar{D}$.   

\medskip

\noindent{\bf Claim 1.} 
\begin{multline*}
A(\mathcal{T})\models x_0(x_1(x_2\dots x_i(x_{i+1}(\dots (x_{n-1}(x_0y))\dots))\dots))\approx\\ 
x_0(x_1(x_2\dots x_{i+1}(x_i(\dots (x_{n-1}(x_0y))\dots))\dots))
\end{multline*} for each $n>3$ and $0<i<n-2$.
\begin{proof}[Proof of Claim 1.]
Assume that the left hand side is nonzero (the dual case for the right hand side will follow by symmetry).  So $y$ is amongst $C,\bar{C},D,\bar{D}$, and as the barred elements are simply duplicates of their unbarred versions, it suffices to assume that $y=C$ or $y=D$.  If $y=D$, then consulting the nonzero products, it is clear that $x_0=x_1=\dots=x_{n-1}=2$ and both sides of the equation take the value $D$.  Now assume that $y=C$.  Then $x_0$ cannot be $H$ because $H\cdot C=D$, forcing $x_{n-1}=x_{n-2}=\dots=x_1=x_0=2$, contradicting $x_0=H$.  So $x_0=1$.  Because the left hand side is nonzero and of the form $x_0\cdot t$ for some term $t$, it follows that $t$ takes the value $C$, which now forces $x_0=x_1=\dots=x_{n-1}=1$, and both sides are again equal.
\end{proof}
Now we recall the structure of some of the large finite subdirectly irreducibles in the variety of $A(\mathcal{T})$: specifically, the algebras denoted  $S_n$ by McKenzie \cite[Definition 4.1]{mck2}.  The elements are $0,a_0,\dots,a_{n-1},b_0,\dots,b_n$, and almost all of the complex array of operations in the signature are constantly equal to $0$, except for multiplication: $a_i\cdot b_{i+1}=b_i$ (and $0$ for all other pairs) and the semilattice $\wedge$, with 
\[
x\wedge y=\begin{cases}
x&\text{ if $x=y$}\\
0&\text{ otherwise,}
\end{cases}
\]
and three further operations, which conveniently are just term functions in~$\cdot$ and~$\wedge$ (on these subdirectly irreducibles, not on $A(\mathcal{T})$).  In the case where $\mathcal{T}$ does not halt, the subdirectly irreducible algebra $S_n$ lies in the variety of~$A(\mathcal{T})$ for every~$n$.  We now consider a distorted version of~$S_n$, which we call~$T_n$: it is an extension of $S_n$ to include some new elements $c_0,\dots,c_{n-1}$ and $d_0,\dots,d_{n-1}$, with the same multiplication pattern, $c_i\cdot d_{i\oplus 1}=d_i$, except now the subscript indices are to be interpreted modulo $n$ (here $\oplus$ denotes addition modulo $n$).  The semilattice $\wedge$ is also flat, and all other operations share their previous definitions, as constantly equal to $0$ or as term functions in $\cdot,\wedge$.  
In either of $S_n$ or $T_n$, we will refer to elements of the form $a_i$ and $c_i$ as being of \emph{$a$-type}, and elements of the form $b_i$ and $d_i$ as being of \emph{$b$-type}.  (So, the $a$-type elements of $S_n$ are $a_0,\dots,a_{n-1}$, whereas the $a$-type elements of $T_n$ are $a_0,\dots,a_{n-1},c_0,\dots,c_{n-1}$.)
\medskip

\noindent{\bf Claim 2.} $T_n\notin \mathsf{HSP}_{\rm fin}(A(\mathcal{T}))$ for all $n>3$.
\begin{proof}[Proof of Claim 2.]
This is because $T_n$ fails the law in Claim 1: assign $x_i$ to $c_i$ and $y$ to $d_1$; the left hand side equals $d_0$, while the right hand side equals $0$.
\end{proof}

We can now use Ehrenfeucht-Fra\"{\i}ss\'e games on the algebras $S_n$ and $T_n$ to show that $\mathsf{HSP}_{\rm fin}(A(\mathcal{T}))$ is not first order definable.  
Note that we do not need to show that there exists a uniformly locally finite class containing $S_n$ and $T_n$, as this only guarantees that an Ehrenfeucht-Fra\"{\i}ss\'e game is a possible approach to proving non first order definability.

Both $S_n$ and $T_n$ have very weak generation power, in the sense that no new indices of elements can be generated from those already present in a set of generators: the only nonzero products are $a_i\cdot b_{i+1}=b_i$ and $c_i\cdot d_{i\oplus 1}=d_i$ in $T_n$.  
These indices can also be used to determine a weak notion of distance: the difference in index suffices, though in the case of $T_n$ this should be cyclic difference within the cyclic component, and we will additionally count elements of the form $a_j$ or $b_j$ as having infinite distance from  elements of type $c_{j'}$ or $d_{j'}$.
At this point, readers familiar with Ehrenfeucht-Fra\"{\i}ss\'e games will see immediately that Duplicator has a winning strategy in the $k$-round game provided $n$ is large enough.  
Readers less familiar are recommended to consult a text such as Libkin~\cite{lib}, where there are many similar examples.  In particular \cite[Figure~3.4]{lib} and accompanying text presents a situation that is almost identical to the present one, save for being graph theoretic and for the fact that in our situation Duplicator must additionally match the correct type of element to that chosen by Spoiler: a choice of $a$-type by Spoiler requires Duplicator to choose $a$-type and likewise for $b$-type.  
To make this more precise, Figure~3.4 of Libkin~\cite{lib} compares two directed graphs, which for our purposes can be denoted~$\overline{S}_n$ and~$\overline{T}_n$ for direct comparison with our algebras $S_n$ and $T_n$.  
The graph $\overline{S}_n$ is a directed path, while $\overline{T}_n$ is the disjoint union of a  directed path with a directed cycle.  
In Libkin, the path lengths are $2m$ and $m$, respectively, however the described strategy works on versions of these digraphs with any larger number of vertices, so there is no loss of generality in fixing large enough $n$ and keeping the path lengths the same.  
We will let the vertices of the directed paths in both graphs be denoted $u_0,\dots,u_{n}$ and the vertices of the directed cycle component of $\overline{T}_n$ be denoted $v_0,\dots,v_{n-1}$; we include one extra vertex $u_n$ in comparison to $v_{n-1}$ because both $S_n$ and $T_n$ contain an ``extra'' element $b_n$ (but not $a_n$).
Duplicator has a strategy to win the $k$-round game on the pair $\overline{S}_n,\overline{T}_n$ provided that $n$ is sufficiently large.  
For Duplicator's strategy on the algebras $S_n$ and $T_n$, first allow for the $0$ of both $S_n$ and $T_n$ to be provided and fixed prior to commencement, so that we may assume that all selections are made from nonzero elements in $S_n$ and $T_n$.  Following the argument in~\cite{lib}, we similarly assume the end points of the paths are provided in advance; so the elements $a_0,b_0,b_n$ in $S_n$ and $T_n$ can be matched prior to commencement. 
For each selection $e_i$ (for $e\in\{a,b\}$) by Spoiler from $S_n$, Duplicator can consult the graph-theoretic strategy for the selection by Spoiler of vertex $u_i$ from $\overline{S}_n$ and make a corresponding selection in $T_n$: if the graph-theoretic strategy is to choose vertex $u_j$ in $\overline{T}_n$, then Duplicator should choose $e_j$ in $T_n$; if the strategy is to choose $v_j$ in $\overline{T}_n$, then Duplicator should choose the corresponding element $f_j$ in $T_n$ of the same type as $e$ (if $e=a$, then $f\coloneq c$ and if $e=b$ then $f\coloneq d$).
The dual scenario, where Spoiler selects from $T_n$ is completely analogous, with Duplicator choosing the element of $S_n$ with the same index as governed by the graph-theoretic strategy, and the same type as Spoiler's selection from $T_n$.
In the algebraic setting, we need to allow for generation of new elements by multiplication, however all products are $0$ (already provided), except for those  of the form $a_i\cdot b_{i+1}=b_i$ and $c_i\cdot d_{i\oplus 1}=d_i$.  
These only occur when, in the corresponding graph setting, the game has encountered adjacencies: $u_i,u_{i+1}$ or $v_i$ and $v_{i\oplus 1}$. 
As directed adjacency is a property that is maintained by Duplicator's winning strategy in the graph-theoretic game, so also are the nonzero instances of algebraic generation when replicating the strategy in the algebraic game.
\end{proof}

\section*{Funding}
The first author is funded by the Australian Research Council \emph{Centre of Excellence for the Mathematical Analysis
of Cellular Systems} CE230100001.  
During the early development of this material, the first author was supported by a research scholarship from the Australian Mathematical Sciences Institute, while the second author was funded by the Australian Research Council Discovery Project DP1094578.

\bibliographystyle{amsplain}

\end{document}